\documentclass[reqno, 11pt]{amsart}

\usepackage{amssymb}
\usepackage[margin=2.5cm]{geometry}

\usepackage{tikz, pgfplots}
\usetikzlibrary{positioning}
\usetikzlibrary{calc}
\usetikzlibrary{patterns}
\usetikzlibrary{patterns.meta}
\pgfplotsset{compat=1.18}

\def\centerarc[#1](#2)(#3:#4:#5)
    {\draw[#1] ($(#2)+({#5*cos(#3)},{#5*sin(#3)})$) arc (#3:#4:#5);}

\newtheorem{thm}{Theorem}[section]
\newtheorem{cor}[thm]{Corollary}
\newtheorem{prop}[thm]{Proposition}
\newtheorem{lem}[thm]{Lemma}
\newtheorem{thmm}{Theorem}

\theoremstyle{definition}
\newtheorem{definition}[thm]{Definition}

\theoremstyle{definition}
\newtheorem{rem}[thm]{Remark}

\newcommand{\R}{\mathbb{R}}
\newcommand{\Z}{\mathbb{Z}}
\newcommand{\C}{\mathbb{C}}
\newcommand{\N}{\mathbb{N}}
\newcommand{\RS}{\widehat{\mathbb{C}}}
\newcommand{\Ree}{\operatorname{Re}}
\newcommand{\Imm}{\operatorname{Im}}
\newcommand{\E}{E_\lambda}
\newcommand{\Orb}{\operatorname{Orb}}

\DeclareMathOperator{\Ind}{Ind}
\DeclareMathOperator{\ind}{ind}
\DeclareMathOperator{\intt}{int}

\begin{document}

\title{Indecomposable continua for unbounded itineraries of exponential maps}
\author{Radosław Opoka}
\address{University of Warsaw, ul. Banacha 2, 02-097 Warsaw, Poland}
\email{r.opoka@student.uw.edu.pl}

\begin{abstract}
We study the dynamics of the exponential maps $\E: \C \longrightarrow \C$ defined by $\E(z) = \lambda e^z$, where $\lambda > \frac{1}{e}$. We prove that for itineraries of a certain form, the set of all points sharing the given itinerary, together with the point at infinity, is an indecomposable continuum in the Riemann sphere. These itineraries contain infinitely many blocks of zeros whose lengths increase, and they may be unbounded. 
We prove that in every such continuum, there exists exactly one point whose $\omega$-limit set contains the repelling fixed point of $\E$. For every other point, the $\omega$-limit set is equal either to the point at infinity, or to the forward orbit of $0$ together with the point at infinity. Thus, we generalize the results of R. Devaney and X. Jarque concerning indecomposable continua for bounded itineraries.
\end{abstract}

\maketitle

%-----------------------
\section{Introduction}
%-----------------------

In this paper we study the iterations of entire functions $\E:\C\longrightarrow\C$ defined by
$$\E(z) = \lambda e^z\textrm{, }\quad\lambda\in\C\setminus\{0\}.$$

The dynamical plane $\C$ can be naturally partitioned into two
disjoint invariant sets: the Fatou set and the Julia set (see \cite{CG, MIL}). A point $z \in \C$ is in the Fatou set of some holomorphic function ${f:\C\longrightarrow\RS}$, if the family $\{f^n\}_{n=1}^\infty$ is defined and normal in the sense of Montel in some neighbourhood of $z$ (i.e.~every sequence of maps from this family has a subsequence converging almost uniformly to a holomorphic map or infinity). This set contains points with stable dynamics in some sense. The Julia set is the complement of the Fatou set. Points in the Julia set exhibit chaotic behaviour. Various results regarding the dynamics of entire maps can be found in \cite{BER}.

To study the dynamics of $\E$ we use symbolic dynamics. We divide the plane into countably many disjoint horizontal strips $R_j$ of height $2\pi$:
$$R_j=\{z\in\mathbb{C}: (2j-1)\pi < \operatorname{Im}(z) \leq (2j+1)\pi\}$$
for $j\in\mathbb{Z}$.
Then we assign a sequence of integers $S(z)=(s_0s_1\ldots)$ to each point $z\in\C$ in such a way that $s_j = k$ if and only if $\E^j(z)\in R_k$. The sequence $S(z)$ is called an itinerary or external address of $z$. 
We say that an itinerary $s=(s_0s_1\ldots)\in\Z^{\N}$ is bounded if there exists $M>0$ such that $|s_k| < M$ for every $k\in\N$. Otherwise, we say that $s$ is unbounded. We are interested in the properties of the sets
$$I_s = \{z\in\mathbb{C} : S(z) = s\},$$
where $s\in\Z^{\N}$.
The answer to the question for which $s$ the set $I_s$ is non-empty was provided by R.~Devaney and M. Krych in \cite{DK} for $\lambda>0$ and by D. Schleicher and J. Zimmer in \cite{SZ} for $\lambda\in\C\setminus\{0\}$. They showed that $I_s$ is non-empty if and only if the itinerary $s$ is exponentially bounded, which means, using the notation from \cite{SZ}, that if $s=(s_0s_1\ldots)$ then
\begin{equation}\label{eq:Expo bounded}
    \exists_{A,x>0} \forall_{k\in\N}\hspace{1mm}|s_k|<AF^k(x),
\end{equation}
where the map $F:\R\longrightarrow \R$ is defined as 
\begin{equation}\label{eq:F}
   F(t) = e^t - 1. 
\end{equation}

In \cite{DK} the structure of the Julia set of $\E$ was studied.
In the case $\lambda \in (0, \frac{1}{e})$ the Julia set of $E_\lambda$ consists of uncountably many disjoint curves. 
Each such curve is a subset of $I_s$ for some exponentially bounded itinerary $s$, is homeomorphic to a half-line $[0,\infty)$ and tends to the point
at infinity. We call these curves, without the endpoints (the point corresponding to the initial point of the half-line), hairs or external rays. All points of the hair are not accessible by curves contained in the complement of the Julia set (see \cite{DG}). In the case of $\lambda \in (0, \frac{1}{e})$ the unique singular value $0$ is contained in the basin of an attracting fixed point.

Some generalizations of the results about the existence of hairs are shown in \cite{6A}.
The existence of hairs for $\E$, $\lambda\in\C\setminus\{0\}$, is shown in \cite{SZ}. In that paper D. Schleicher and J. Zimmer proved that any escaping point i.e. $z\in\C$ such that $\E^n(z)\xrightarrow{n\to\infty}\infty$, is an element of the unique hair, or is a landing point (endpoint) of the unique hair, or maps after finitely many iterations on the hair on which the singular value 0 lies. The third case only occurs when the singular value escapes.
Moreover, the existence of similar hairs is a phenomenon that occurs in a broader context, not only for exponential maps (see \cite{KB, RRRS}). 
M. Viana showed that the hairs for the complex exponential family are $C^\infty$-smooth curves (see \cite{VIA}). 
On the other hand, in \cite{COM}, P. Comdühr showed an example of an entire map with the Julia set consisting of hairs, which are nowhere-differentiable curves.

We obtain different dynamics if we consider the situation when the singular value $0$ tends to the point at infinity, which happens e.g.~for $\lambda \in (\frac{1}{e}, \infty)$. In this case the Julia set is equal to the entire plane $\C$ (see \cite{D, MIS}). In \cite{D} R. Devaney proved that for $\lambda>\frac{1}{e}$ there is a natural compactification of the set $\{z\in\C: \forall_{n\in\N}\quad 0\leqslant\Imm(\E^n(z))\leqslant\pi\}$, which is an indecomposable continuum (an indecomposable continuum is a continuum that cannot be written as the union of two proper subcontinua). This set is a subset of $I_s$ for $s=(000\ldots)$.
In \cite{DJ} R. Devaney and X. Jarque analyse itineraries of the form 
$$s = 0_{n_0}t_{0}0_{n_1}t_{1}0_{n_2}t_{2}\ldots,$$
where $t_{0}, t_{1}, \ldots$ is a sequence of blocks of integers in which at least one term of every block is non-zero, and $0_{n_k}$ is a block of zeros of length $k$. By a block we mean a finite sequence of symbols. They proved that if $\lambda>\frac{1}{e}$ and $s$ is a bounded itinerary of the above form, then there exists a sequence of natural numbers $(n_q)_{q\in\N}$, such that the set $I_s\cup\{\infty\}$ is an indecomposable continuum in the Riemann sphere. 
Moreover, 
$$I_s\cup\{\infty\} = \overline{\gamma_s}\cup\{\infty\} = A(\gamma_s),$$ where $\gamma_s:(0,+\infty)\longrightarrow\C$ is a curve such that $\lim_{\eta\to +\infty}\gamma_s(\eta) = \infty$, consisting of escaping points, and $A(\gamma)$ is the accumulation set of the curve $\gamma:(0,+\infty)\longrightarrow\C$, which is the set of all possible limits $\lim_{\eta_n\to 0}\gamma(\eta_n)$. We also call $\gamma_s$ a hair. The above phenomenon arises when a hair in the Julia set becomes so entangled that it accumulates everywhere on itself. 
R. Devaney and X. Jarque in \cite{DJ} proved also that if $\lambda>\frac{1}{e}$ and $s$ is a bounded itinerary then there exists exactly one element $z_s$ of $I_s$ with a bounded orbit, where a (forward) orbit of $z$ under some function $f$ is a set $\Orb(z)=\{z,f(z),f^2(z),\ldots\}$. For every other point $z$ in $I_s$ one of the two occurs:
\begin{itemize}
    \item $z\in\gamma_s$, then $\E^n(z)\to\infty$ with $n\to +\infty$;
    \item $z\in I_s\setminus(\gamma_s\cup z_s)$, then the $\omega$-limit set of $z$ is equal to $\Orb(0)\cup\{\infty\}$.
\end{itemize}
The set of elements of $I_s$ satisfying the second condition may be empty. In this case, the hair $\gamma_s$ lands at $z_s$ ($z_s$ is an endpoint of $\gamma_s$) and the closure of the hair with the point at infinity cannot be an indecomposable continuum.

The existence of indecomposable continua was proved in many other cases.
For $\E$ indecomposable continua occur in the Misiurewicz case, that is when the singular value $0$ has a finite trajectory under iteration, e.g. for $\lambda=2\pi i$ (see \cite{DJR}). In \cite{REM} L. Rempe showed that, if $\E$ is an exponential map whose singular value $0$ is on a hair or is the landing point (endpoint) of such hair, then there exists a hair $\gamma$ whose accumulation set (in
the Riemann sphere) is an indecomposable continuum containing $\gamma$. In the case of $\lambda >\frac{1}{e}$ constructed itineraries in the proof of Rempe's theorem are again bounded, as in \cite{DJ}.
In \cite{PZ} Ł. Pawelec and A. Zdunik showed that the Hausdorff dimension of each continuum constructed in \cite{DJ, DJR, REM} is equal to 1. In \cite{FZ} J. Fu and G. Zhang proved the existence of hairs with a bounded accumulation set, which can be either an indecomposable continuum or a Jordan arc. Such accumulation set is a subset of $I_s$ for some bounded $s$. The construction of such hairs is carried out for $\lambda = 2\pi i$. Note that the itinerary is defined differently in \cite{DJR, FZ, REM} than in \cite{DJ}. 

In \cite{LYU} M. Lyubich described ergodic properties of $E_1(z)=e^z$ and proved that for the map $E_1$ the $\omega$-limit set of almost every $z\in\C$ is $\Orb(0) \cup \{\infty\}$. 
This result can be extended to the case $\lambda > \frac{1}{e}$ (see \cite{ROM}).

In this paper we are interested in the case $\lambda>\frac{1}{e}$. Let us recall that results in \cite{DJ, DJR, FZ} are provided assuming that itineraries are bounded. We abandon this assumption and we show the existence of indecomposable continua in the case of unbounded itineraries. 
Let us recall that a block is a finite sequence of integers, and $0_k$ is a block of zeros, which length is $k\in\N$. The main result is stated as follows.

\begin{thmm}\label{thm:A}
    Consider a map $\E(z)=\lambda e^z$, $\lambda>\frac{1}{e}$. Let $t_{0}, t_{1}, \ldots$ be a sequence of finite blocks of integers, such that for every $m\in\N$ at least one term in the block $t_m$ is non-zero. Then, there exists a sequence of natural numbers $(n_q)_{q\in\N}$ such that if
    $$s = 0_{n_0}t_{0}0_{n_1}t_{1}0_{n_2}t_{2}\ldots,$$
    then $I_s\cup\{\infty\} = \overline{\gamma_s}\cup\{\infty\}=A(\gamma_s)$ and $I_s\cup\{\infty\}$ is an indecomposable continuum in the Riemann sphere.
\end{thmm}

This theorem is a generalization of the mentioned theorem of R. Devaney and X. Jarque (\cite[Theorem 5.1]{DJ}).

For $\lambda>\frac{1}{e}$ the map $\E$ has two repelling fixed points in $R_0$: $q_+$ and $q_-$. The imaginary part of $q_+$ is positive and the imaginary part of $q_-$ is negative (see \cite{D}).
The second result of this paper is related to another result of R. Devaney and X. Jarque (\cite[Theorem 6.1 and Corollary 6.2]{DJ}).

\begin{thmm}\label{thm:B}
    Consider a map $\E(z)=\lambda e^z$, $\lambda>\frac{1}{e}$. Let $t_{0}, t_{1}, \ldots$ be a sequence of finite blocks of integers, such that for every $m\in\N$ at least one term in the block $t_m$ is non-zero.
    Then, there exists a sequence of natural numbers $(n_q)_{q\in\N}$ such that if
    $$s = 0_{n_0}t_{0}0_{n_1}t_{1}0_{n_2}t_{2}\ldots,$$
    then there exists exactly one point $z_s\in I_s$ such that
        \begin{itemize}
        
        \item the $\omega$-limit set of $z_s$ contains at least one of the points $q_+$, $q_-$;
        
        \item if $z\in I_s\setminus(\gamma_s\cup\{z_s\})$, then the $\omega$-limit set of $z$ is equal to $\Orb(0)\cup\{\infty\}$.
\end{itemize}
Every other point from the set $I_s$ lies on the curve $\gamma_s$ and the $\omega$-limit set of this point is $\{\infty\}$.
\end{thmm}

The structure of the paper is as follows. In Section 2 we present the notation and the initial definitions and theorems that we refer to in this paper. In Section 3 we introduce itineraries with linear growth and we prove their properties. In Section 4 we introduce targets. In Section 5 we construct itineraries for which hairs pass twice through a proper target. The properties shown in these two sections allow us to prove the existence of hair that accumulates on itself. In Section 6 we prove Theorem \ref{thm:A}. In Section 7 we prove Theorem \ref{thm:B}.

\section*{Acknowledgments}

This research was funded in whole or in part by the National Science Centre, Poland, grant no. 2023/51/B/ST1/00946. For the purpose of Open Access, the author has applied a CC-BY public copyright licence to any Author Accepted Manuscript (AAM) version arising from this submission.

I would like to thank my supervisor, Krzysztof Bara\'nski, for his help and support. I would also like to thank the Christian-Albrechts-Universit\"at zu Kiel for their hospitality and Walter Bergweiler for the useful discussions during my research stay.

%-----------------------
\section{Notation and preliminaries}
%-----------------------

 In this article we use the notation $\N = \{0, 1, 2, \ldots\}$. We denote the Riemann sphere by $\RS = \C\cup\{\infty\}$. By $B(z,r)$ we denote an open disc with the centre at $z\in\C$ and radius $r>0$. We denote the closure of set $A$ in $\C$ by $\overline{A}$. We indicate when closures in other spaces are considered (e.g. in $\RS$). 
 We say that the set $A\subseteq \C$ does not separate the plane if the set $\C\setminus A$ is connected. 
 
 Let $\varphi: [0,+\infty) \longrightarrow \RS$ be a continuous and injective curve. We say that the curve $\varphi$ accumulates on the set $A\subseteq\RS$, if for every $z\in A$, there exists a sequence $(\eta_j)_{j=0}^\infty$ such that $\lim_{j\to\infty} \eta_j=+\infty$ and $\lim_{j\to\infty}\varphi(\eta_j)=z$. The curve $\varphi: [0,+\infty) \longrightarrow \RS$ accumulates on itself when it accumulates on its image. 
 
 For an infinite sequence $s\in\Z^\N$ we use the notation $s=(s_0s_1\ldots)$ when we want to list terms of $s$, and similarly for a finite sequence $p$ we use the notation $p=(p_0p_1\ldots p_n)$. On the other hand, if $p$ is a finite sequence of integers and $t$ is a finite or infinite sequence of integers, we use a concatenation of $p$ and $t$ to denote a sequence $s$ in which the sequence $p$ is followed by the sequence $t$: $s=pt=p_0p_1\ldots p_n t$. We set the value of a product $\prod_{k=1}^0 a_n$ to be equal $1$ and the value of a sum $\sum_{k=1}^0 a_n$ to be equal $0$.
 
Let $\ind X$ denote the small inductive dimension of $X$, $\Ind X$ denote the large inductive dimension of $X$ and $\dim X$ denote the topological dimension of $X$ (see \cite{ENG}).
Let us present two theorems about these dimensions. 
        
    \begin{thm}\cite[Theorem 1.7.7]{ENG}\label{thm:Dimension 2}
    For every separable metric space $X$ we have
    $\ind X = \Ind X = \dim X$.    
    \end{thm}

    \begin{thm}\cite[Theorem 1.8.10]{ENG}\label{thm:Dimension}
    A subset $X$ of $n$-dimensional Euclidean space $\mathbb{R}^n$ satisfies $\ind X = n$ if and only if the interior of $X$ is non-empty.
    \end{thm}

Due to Theorem \ref{thm:Dimension 2}, we refer to the topological dimension of a set $A\subseteq\C$ without distinguishing the type of the dimension.

In this paper, as in \cite{DJ}, we refer to the following theorem of S. Curry (see \cite{C}), which is the basis for the construction presented in this paper.
 
\begin{thm}\cite[Theorem 8]{C}\label{thm:Curry}
Let $X$ be a one-dimensional continuum contained in the plane, which is the closure of an injective curve that accumulates on itself. Assume that $X$ does not separate the plane. Then $X$ is an indecomposable continuum.
\end{thm}

It is known that for $\lambda>\frac{1}{e}$ the Julia set of $\E$ is equal to $\C$. (see, for instance, \cite{D, MIS}). Combining this result with Montel's theorem (see \cite{CG}) we obtain the corollary.

\begin{cor}\label{cor:Montel}
   Let $\lambda>\frac{1}{e}$ and $U\subseteq \C$ be a non-empty, open domain. Then 
   $$\C\setminus\{0\} \subseteq\bigcup_{n\in\N}\E^n(U).$$
\end{cor}

Now we state Koebe's theorem (see, for instance, \cite[Theorem 1.3]{POM}).
\begin{thm}\label{thm:Koebe}
Suppose that $f$ is a univalent holomorphic function in
the disc $B(a, r)$ and let $0 < \eta < 1$. Then for $z\in B(a, \eta r)$ we have 
$$\frac{|f'(a)|\eta r}{(1+\eta)^2} \leqslant |f(z)-f(a)| \leqslant \frac{|f'(a)|\eta r}{(1-\eta)^2}.$$
\end{thm}

We present the following lemma. The proof of this lemma for $\lambda = 1$ can be found in \cite[Lemma 1]{MIS}. A proof for all $\lambda > \frac{1}{e}$ is provided, for instance, in \cite[Lemma 6.5]{DJ}.
    \begin{lem}\label{lem:Mis}
        For every $z\in\C$ and every $n\in\N$ the following inequality holds
        $$|\Imm(\E^n(z))| \leqslant |(\E^n)'(z)|.$$
    \end{lem}

Let us finish this section with the following definitions.

\begin{definition}
    Let the inverse maps of $\E$ be
    $L_{\lambda, k}: \C\setminus\{z\in\C: \Imm(z) = 0\textrm{, }\Ree(z) \leqslant 0\} \longrightarrow \C$ defined as $$L_{\lambda, k}(z) = L_k\left(\frac{z}{\lambda}\right),$$ where $L_k$ is a branch of the logarithm defined on $\C\setminus\{z\in\mathbb{C}: \Imm(z) = 0\textrm{, }\Ree(z) \leqslant 0\}$, which values belong to $R_k$. 
\end{definition}

\begin{definition}
    Let $\sigma:\Z^{\N} \longrightarrow \Z^{\N}$  be the shift map defined by $$\sigma((s_0s_1s_2\ldots)) = (s_1s_2s_3\ldots).$$
\end{definition}

%-----------------------
\section{Unbounded itineraries with linear growth}
%-----------------------

In this section, we present a special subset of itineraries -- the set of itineraries with linear growth. This set contains both bounded and unbounded itineraries. Later, we use the properties of these itineraries to construct indecomposable continua.

Let $\Omega$ be a set of itineraries that do not end with an infinite sequence of zeros, i.e.
    $$
    \Omega = \{(s_j)_{j=0}^\infty\in\mathbb{Z}^{\mathbb{N}}: \forall_{j\in\mathbb{N}} \hspace{1mm} \exists_{m>j}\hspace{1mm} s_m\neq 0\}.
    $$
In the whole paper we consider only itineraries from $\Omega$.
Let the set $\Gamma_s$ be the set of escaping points with itinerary $s$, i.e.,
$$\Gamma_s = \{z\in I_s: \E^n(z) \xrightarrow[]{n \to \infty} \infty\}.$$
According to \cite[Theorem 4.2, Corollary 6.9]{SZ} and the fact that $\lambda\in\R$, for every exponentially bounded itinerary $s\in\Omega$ (see \eqref{eq:Expo bounded}), the set $\Gamma_s$ is an image of the injective curve of the form $\psi:(0,+\infty)\longrightarrow\C$ or $\psi:[0,+\infty)\longrightarrow\C$, such that $\lim_{\eta\to +\infty}\psi(\eta) = \infty$. 
In both cases we denote 
$$\gamma_s = \psi\big|_{(0,+\infty)}$$ 
and we call $\gamma_s$ a hair (ray) associated with the itinerary $s$.
If for the hair $\gamma_s$ there exists the limit $w=\lim_{\eta\to 0} \gamma_s(\eta)$ then we say that the hair $\gamma_s$ lands at $w$, and we call $w$ the endpoint of $\gamma_s$.
If $\Gamma_s$ is an image of the curve of the form $\psi:[0,+\infty)\longrightarrow\C$ then the hair $\gamma_s$ lands at $\psi(0)$. If $\Gamma_s$ is of the form $\psi:(0,+\infty)\longrightarrow\C$ then $\gamma_s$ may or may not have an endpoint. We define the accumulation set $A(\gamma_s)\subseteq\RS$ of a hair $\gamma_s$ as the set of all possible limits $\lim\gamma_s(\eta_n)$, where $(\eta_n)_{n\in\N}$ is such that $\eta_n\to 0$. If the hair $\gamma_s$ lands at $w$, then $A(\gamma_s)=\{w\}$.

According to \cite[Proposition 3.4]{SZ} for every exponentially bounded itinerary $s$ there exists $\eta_s$ such that for every $\eta\geqslant \eta_s$, with proper parametrization of $\gamma_s$, we have
\begin{equation}\label{eq:Functional equation}
    \E(\gamma_s(\eta)) = \gamma_{\sigma(s)}(F(\eta)).
\end{equation}
Moreover, if there are $x> 0$ and $A\geqslant \frac{1}{2\pi}$ such that $|s_k| < AF^k(x)$ for every $k\in\N$, then we can set $\eta_s = x + 2 \ln(\ln(\lambda) + 3)$ and for $\eta\geqslant\eta_s$ we have
\begin{equation}\label{eq:Hair parametrization}
    \gamma_s(\eta) = \eta - \ln(\lambda) + 2\pi is_0 + r_s(\eta),
\end{equation}
where $r_s$ satisfies
\begin{equation}\label{eq:Remainder}
    |r_s(\eta)| < 2e^{-\eta}(\ln(\lambda) + 2 + 2\pi |s_1| + 2\pi AC)
\end{equation}
with a universal constant $C>0$.

\begin{definition}
    Let $M\in\N$ and $p\in\N$. Let the set $\Sigma_M^p$ be defined as
    $$
    \Sigma_M^p=\{(s_j)_{j=0}^\infty\in\Omega: \forall_{j\in\mathbb{N}} \hspace{1mm} |s_j|\leqslant M + jp\}.
    $$
\end{definition}
All itineraries from $\Sigma_M^p$ are exponentially bounded. For $p\geqslant 1$ the set $\Sigma_M^p$ contains itineraries of both types, bounded and unbounded. The set $\Sigma_M^0$ is a set of itineraries bounded by constant $M$ that do not end with all zeros. This is the set of itineraries analysed in the paper \cite{DJ}. 
For $n\in\N$, let us introduce the notation 
$$M_n = M + np.$$

\begin{lem}\label{lem:Dzeta}
    For all $M,p\in\N$ there exists $\zeta(M,p)\in\R$ such that for every $s\in\Sigma_M^p$ there exists $\theta_s\geqslant\eta_s$ such that for all $\eta\geqslant\theta_s$ we have
    $$\Ree(\gamma_s(\eta)) \geqslant \zeta(M,p)$$
    and the equality holds only for $\eta=\theta_s$.
\end{lem}
\begin{proof}
    For fixed $M$ and $p$ there exist $x> 0$ and $A\geqslant\frac{1}{2\pi}$ such that for every $n\in\N$ we have $M_n < AF^n(x)$. Therefore, for every $s\in\Sigma_M^p$ and every $k\in\N$ we have $|s_k| < AF^k(x)$. Thus, there exists $\eta^* = x+2\ln(\ln(\lambda)+3)$ such that we can set $\eta_s=\eta^*$ for every $s\in\Sigma_M^p$. Moreover, there exists a constant $D$ such that $|r_s(\eta)| < D$ for every $s\in\Sigma_M^p$ and every $\eta \geqslant \eta^*.$ Therefore, we can choose $\zeta(M,p)\in\R$ large enough such that for every $s\in\Sigma_M^p$ we have
    $$ \{z\in\C: \Ree(z) < \zeta(M,p)\} \ \cap \gamma_s([\eta^*,\infty)) \neq \varnothing.$$
    Thus, for every $s\in\Sigma_M^p$ there exists $x$ such that $\Ree(\gamma_s(x)) = \zeta(M,p)$. The largest such $x$ satisfy $\Ree(\gamma_s(x)) = \zeta(M,p)$ and $\Ree(\gamma_s(\eta)) > \zeta(M,p)$ for all $\eta>x$. Thus, we can set $\theta_s$ as the largest $x$ which satisfy $\Ree(\gamma_s(x)) = \zeta(M,p)$.
\end{proof}

Let $s\in\Sigma_M^p$. We call
$$\omega_{s,\zeta(M,p)} = \gamma_s([\theta_s,+\infty))$$ 
the tail of the hair $\gamma_s$. If it is needed, we can take a larger $\zeta(M,p)$.

From now on, we fix $M\in\N$ and $p\in\Z_+$. Therefore, we obtain fixed $\zeta=\zeta(M,p)$. Later, we impose more conditions on $\zeta$ that may cause this number to increase. For every $n\in\Z_+$ we have $\{\sigma^n(s):s\in\Sigma_M^p\}=\Sigma_{M_n}^p$. We set $\zeta_n=\zeta(M_n,p)$.
Every tail $\omega_{s,\zeta}=\gamma_s([\theta_s,+\infty))$ is a curve which lies in the half-plane $\{z\in\C: \Ree(z) \geqslant \zeta\}$, $\lim_{\eta\to +\infty} \Ree(\gamma_{s}(\eta)) = +\infty$, and begins at a point with the real part equal to $\zeta$. We use the symbol $\gamma_s$ to denote the image of the curve $\gamma_s$ unless this leads to confusion.
We have $$\omega_{s,\zeta}\subseteq\gamma_s\subseteq I_s\subseteq R_{s_0}.$$

\begin{prop}\label{prop:Re increases}
    If $\zeta$ is chosen large enough then for every $s\in\Sigma_M^p$ and every $z\in\omega_{s,\zeta}$ we have $\Ree(\E^n(z)) < \Ree(\E^{n+1}(z))$ for every $n\in\N$.
\end{prop}

\begin{proof}
    Let $x\geqslant 1$ and $A\geqslant\frac{1}{2\pi}$ be such that $M_k < A F^k(x)$ for every $k\in\N$. Thus, for every $s\in\Sigma_M^p$ and every $k\in\N$ inequality $|s_j| < AF^k(x)$ holds.
    We have $\omega_{s,\zeta}=\gamma_s([\theta_s,+\infty))$ for some $\theta_s \geqslant x + 2\ln(\ln(\lambda) + 3)$. Due to \eqref{eq:Functional equation}, for every $s\in\Sigma_M^p$ and every $n\in\N$ the curve $\E^n(\omega_{s,\zeta})$ is of the form $\gamma_{\sigma^n(s)}([F(\theta_s), +\infty))$. For every $s\in\Sigma_M^p$ and every $n,k\in\N$ we have 
    \begin{equation}\label{eq:Const for shift}
        (A+np)F^k(x) = AF^k(x) + npF^k(x) > M_k + np = M_{n+k}.
    \end{equation}
    Thus for every $s\in\Sigma_M^p$, if $\sigma^n(s)=u=(u_0u_1\ldots)$ then $|u_k| < (A+np)F^k(x)$ for every $k\in\N$. Therefore, based on \eqref{eq:Remainder}, for every $\eta > F^n(\theta_s)$ we have 
    \begin{equation}\label{eq:Remainder of parametrization}
    \begin{aligned}
        |r_{\sigma^n(s)}(\eta)| & < 2e^{-F^n(\theta_s)}(\ln(\lambda) + 2 + 2\pi(M+(n+1)p) + 2\pi (A+np)C) \\
        & < 2e^{-F^n(x)}k_n,
    \end{aligned}
    \end{equation}
    where $k_n = \ln(\lambda) + 2 + 2\pi(M+(n+1)p) + 2\pi (A+np)C$. Then $(k_n)_{n\in\N}$ is an arithmetic sequence: $k_{n+1}-k_n = 2\pi p (1+C)$. Therefore, for $x$ large enough, $|r_{\sigma^n(s)}(\eta)| < 2e^{-F^n(x)}k_n \to 0$ monotonically with $n\to +\infty$. Let us choose $\zeta$ large enough such that for every $s\in\Sigma_M^p$ and every $n\in\N$ inequality $|r_{\sigma^n(s)}(\eta)| < 1$ holds for every $\eta > F^n(\theta_s)$, and that $F(\theta) - \theta > 2$ for every $\theta > \theta_s$. Thus, from \eqref{eq:Hair parametrization} we know that for every $s\in\Sigma_M^p$, every $z\in\omega_{s,\zeta}$ and every $n\in\N$, the point $\E^n(z)$ is inside the disc $B(w_1, 1)$ and the point $\E^{n+1}(z)$ is inside the disc $B(w_2, 1)$, where $w_1$ and $w_2$ are points with imaginary part equal to some even multiple of $\pi$ and whose real parts satisfy $\Ree(w_2) = F(\Ree(w_1)) > \Ree(w_1) + 2$. Therefore $\Ree(\E^{n+1}(z)) > \Ree(\E^n(z))$.
\end{proof}

Due to \eqref{eq:Const for shift} there exists $\eta^*$ such that for every $s\in\bigcup_{n\in\N}\Sigma_{M_n}^p$ we can set $\eta_s=\eta^*$. Combining this with \eqref{eq:Remainder of parametrization}, we get that for every $s\in\Sigma_M^p$, every $n\in\N$ and every $\eta\geqslant\eta^*$ the inequality $|r_{\sigma^n(s)}(\eta)| < 2e^{\eta^*}k_n$ holds. Thus, we can choose the $\zeta_n$ numbers so that they form an arithmetic sequence. Therefore, we can represent $\gamma_s$ as
$$\gamma_s = \displaystyle{ \bigcup_{n=1}^\infty L_{\lambda, s_0} \circ L_{\lambda, s_1} \circ \ldots \circ L_{\lambda, s_{n-1}} \left(\omega_{\sigma^{n}(s),\zeta_{n}} \right)}.$$

For sufficiently large $\zeta$, the function $\E$ maps the line $\{z\in\mathbb{C}:\Ree(z) = \zeta\}$ to the circle with radius $\lambda e^{\zeta}$ centred at $0$, which intersects the lines defined by $\Imm(z) = \pm (2M_0+1)\pi$ at points with real parts greater than $\zeta$. From now on, let $\zeta$ be large enough that the above property holds and the thesis of Proposition \ref{prop:Re increases} holds.

For $s \in \Sigma_M^p$ let $\hat{s}$ be such that $\hat{s}\in\sigma^{-1}(s)$ and $\hat{s}\in\Sigma_M^p$ e.g. $\hat{s} = (0s_0s_1s_2\ldots)$. Our choice of $\zeta$ implies that the set $\E(\omega_{\hat{s},\zeta})$ is properly contained in $\omega_{s,\zeta}$, because $\E(\omega_{\hat{s},\zeta})$ lies in the strip $R_{s_0}$ and to the right from the circle with radius $\lambda e^\zeta$ centred at $0$. 
We can therefore introduce the following definition.

\begin{definition}\label{def:Alpha}
    Let $s\in\Sigma_M^p$. Let $\alpha_{s,\zeta}=\omega_{s,\zeta}\setminus \E(\omega_{\hat{s},\zeta})$, where $\hat{s}$ is any of the elements of $\sigma^{-1}(s)\cap\Sigma_M^p$. We call $\alpha_{s,\zeta}$ the base of the tail $\omega_{s,\zeta}.$ 
\end{definition}

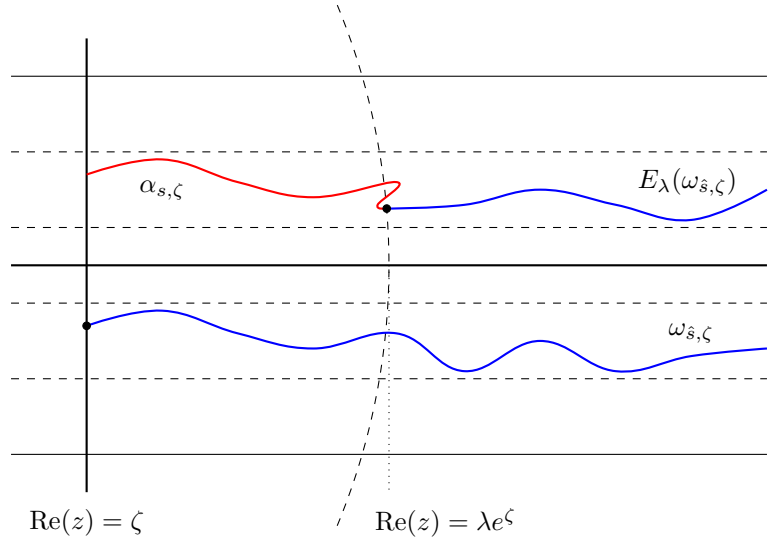
\begin{figure}[h]
    \centering
    \begin{tikzpicture}[font=\small]

        \draw [thick] (-1,0) -- (9,0);
        \draw [dashed] (-1,0.5) -- (9,0.5);
        \draw [dashed] (-1,-0.5) -- (9,-0.5);
        \draw [dashed] (-1,1.5) -- (9,1.5);
        \draw [dashed] (-1,-1.5) -- (9,-1.5);    
        \draw (-1,2.5) -- (9,2.5);
        \draw (-1,-2.5) -- (9,-2.5);

        \draw [thick] (0,-3) -- (0,3);
        \draw node at (0,-3.4) {$\Ree(z)=\zeta$};

        \centerarc[dashed](-5,0)(-22.5:22.5:9)
        
        \draw [dotted] (4,0) -- (4,-3);
        \draw node at (4.75,-3.4) {$\Ree(z)=\lambda e^\zeta$};

        \draw [thick, red] plot [smooth, tension=0.7] coordinates {(0,1.2) (1,1.4) (2, 1.1) (3, 0.9) (4.1, 1.1) (3.85, 0.8) (3.95, 0.75)};
        \draw [thick, blue] plot [smooth, tension=0.7] coordinates {(3.95, 0.75) (5, 0.8) (6, 1) (7, 0.8) (8, 0.6) (9, 1)};

        \draw [thick, blue] plot [smooth, tension=0.7] coordinates {(0,-0.8) (1,-0.6) (2, -0.9) (3, -1.1) (4.1, -0.9) (5, -1.4) (6, -1) (7, -1.4) (8, -1.2) (9, -1.1)};

        \filldraw (3.97, 0.75) circle (0.05);
        \filldraw (0,-0.8) circle (0.05);
        
        \draw node at (1,1) {$\alpha_{s,\zeta}$};
        \draw node at (7.95,1.1) {$\E(\omega_{\hat{s},\zeta})$};
        \draw node at (8,-0.9) {$\omega_{\hat{s},\zeta}$};

    \end{tikzpicture}
    \caption{The tail $\omega_{s,\zeta}$ and its base $\alpha_{s,\zeta}$.}
    \label{fig:Omega}
\end{figure}

We denote $\omega_{s,\zeta}$ and $\alpha_{s,\zeta}$ by $\omega_s$ and $\alpha_s$, respectively, when it does not lead to confusion. 

\begin{rem}\label{rem:Alpha}
    Let us notice that we have $\alpha_s = \gamma_s([\theta_s, F(\theta_{\hat{s}})))$. This is due to the \eqref{eq:Functional equation}.
\end{rem}

%-----------------------
\section{Accumulating on itself -- targets}
%-----------------------

In this section, we introduce the concept of a target. The properties of targets enable us to produce the very tangled hairs in the subsequent section. The accumulation set of such hairs turns out to be an indecomposable continuum. The definition of the target appears in \cite{DJ} but we modify it. We have to take into account the fact that we are interested in unbounded itineraries. Let us recall that we have fixed constants $M\in\N$ and $p\in\Z_+$. Therefore, we have fixed constant $\zeta$, but we impose some new restrictions on $\zeta$, which may cause it to increase.

For every $n\in\N$ let $a_n$ and $b_n$ be defined as
$$a_n = \inf\{\Ree(z): z\in\E^n\left(A_M^p\right)\} \hspace{1mm}\textrm{ and } \hspace{1mm} b_n = \E^{n+1}(\zeta) + 1,$$
where
$A_M^p = \bigcup_{s\in\Sigma_M^p}\alpha_s.$
From the above definition, it follows that $a_0=\zeta$. We have $a_n < a_{n+1}$ for every $n\in\N$. It follows from Proposition \ref{prop:Re increases}. Namely, for every $s\in\Sigma_M^p$ and every $z\in \E^n(\alpha_s)$ we have $\Ree(\E(z)) > \Ree(z)$.

\begin{lem}\label{lem:Re less than b}
  Provided $\zeta$ is large enough, for every $n\in \mathbb{N}$ and every $z\in A_M^p$, we have $\operatorname{Re}(\E^n(z)) \leqslant b_n$.
\end{lem}

\begin{proof}
For every $s\in\Sigma_M^p$ we have $\omega_s= \gamma_s([\theta_s, +\infty)])$. For every $n\in\N$ the set $\E^n(\omega_s)$ is of the form $\E^n(\omega_s) = \gamma_{\sigma^n(s)}([F^n(\theta_s),+\infty))$. We have also $\E^n(\alpha_s) = \gamma_{\sigma^n(s)}([F^n(\theta_s), F^{n+1}(\theta_{\hat{s}})))$, where $\hat{s}$ is as in Definition \ref{def:Alpha}. 

Let us prove that for every $s\in\Sigma_M^p$ and every $n\in\N$ we have 
$$\Ree(\gamma_{\sigma^n(s)}(F^{n+1}(\theta_{\hat{s}}))) \leqslant \E^{n+1}(\zeta).$$
For $n=0$ the above inequality is of the form $\Ree(\gamma_{s}(F(\theta_{\hat{s}}))) \leqslant \E(\zeta)$. This is true because $\gamma_{s}(F(\theta_{\hat{s}})) = \E(\gamma_{\hat{s}}(\theta_{\hat{s}}))$, so $\gamma_{s}(F(\theta_{\hat{s}}))$ lies on the circle with centre at $0$ and radius $\E(\zeta)$. Now, let us assume that $\Ree(\gamma_{\sigma^k(s)}(F^{k+1}(\theta_{\hat{s}}))) \leqslant \E^{k+1}(\zeta)$ holds for some $k\in\N$. Then, 
\begin{align*}
    \Ree(\gamma_{\sigma^{k+1}(s)}(F^{k+2}(\theta_{\hat{s}}))) & \leqslant |\gamma_{\sigma^{k+1}(s)}(F^{k+2}(\theta_{\hat{s}}))| = |\E(\gamma_{\sigma^{k}(s)}(F^{k+1}(\theta_{\hat{s}})))| = \\
    & = \lambda \exp(\Ree(\gamma_{\sigma^{k}(s)}(F^{k+1}(\theta_{\hat{s}})))) \leqslant \\
    & \leqslant \lambda \exp(\E^{k+1}(\zeta)) = \E^{k+2}(\zeta),
\end{align*}
which completes the inductive proof.

Again, as in the proof of Proposition \ref{prop:Re increases}, by choosing $\zeta$ large enough we can bound $|r_{\sigma^n(s)}(\eta)|$ by the same constant for every $n$.  Let this constant be $\frac{1}{2}$. Now, let $z\in\E^n(\alpha_s)$. Then $z=\gamma_{\sigma^n(s)}(\eta)$ for some $\eta\in [F^n(\theta_s), F^{n+1}(\theta_{\hat{s}}))$. We have $|\Ree(z) - \eta| < \frac{1}{2}$, so $\Ree(z) < F^{n+1}(\theta_{\hat{s}}) + \frac{1}{2}$. Similarly, we have $|\Ree(\gamma_{\sigma^n(s)}(F^{n+1}(\theta_{\hat{s}}))) - F^{n+1}(\theta_{\hat{s}})| < \frac{1}{2}$, so $\Ree(\gamma_{\sigma^n(s)}(F^{n+1}(\theta_{\hat{s}}))) > F^{n+1}(\theta_{\hat{s}}) - \frac{1}{2}$. Thus, 
$$\Ree(z) - \Ree(\gamma_{\sigma^n(s)}(F^{n+1}(\theta_{\hat{s}}))) < 1.$$
So we have 
$$\Ree(z) - \E^{n+1}(\zeta) < 1,$$
which ends the proof as $b_n = \E^{n+1}(\zeta) + 1$.
\end{proof}

\begin{definition}
    Let $K\in\N$ and $a,b\in\R$ be such that $\zeta\leqslant a<b$. Let 
    $$ V(a,b, K) = \{ z\in \mathbb{C} : a-1 \leqslant \operatorname{Re}(z) \leqslant b+1,\textrm{ } |\operatorname{Im}(z)| \leqslant (2K+1)\pi\}.$$
\end{definition}

Lemma \ref{lem:Re less than b} and the definition of $a_n$ gives the following conclusion.
\begin{cor}\label{cor:Alpha in V}
    For every $s\in\Sigma_M^p$ we have $\E^n(\alpha_s)\subseteq V(a_n,b_n,M_n)$.
\end{cor}

\begin{prop}\label{prop:E(V) covers V}
    Provided $\zeta$ is large enough, we have
    $$\E\left(V(a_{n},b_{n+k}, M_n)\right) \supseteq V(a_{n+1},b_{n+k+1}, M_{n+1})$$
    for every $n\in\N$ and every $k\in\N$.
\end{prop}

To prove the above proposition we need to introduce a new definition and prove some lemmas.

\begin{definition}
    For $r>0$ let $\kappa(r)$ denote the circle with its centre at $0$ and radius~$r$.
    Let $y > 0$ and $0<\delta<r$. We say that the circle $\kappa(r)$ is $\delta$-vertical at height $y$ if
    $$\kappa(r) \cap \{z\in\C: \Imm(z) = y\} \cap \{z\in\C: \Ree(z) \geqslant 0\} \neq \varnothing$$
    and the only point $w$ belonging to this intersection satisfies
    $$\Ree(w) \geqslant r-\delta.$$
\end{definition}

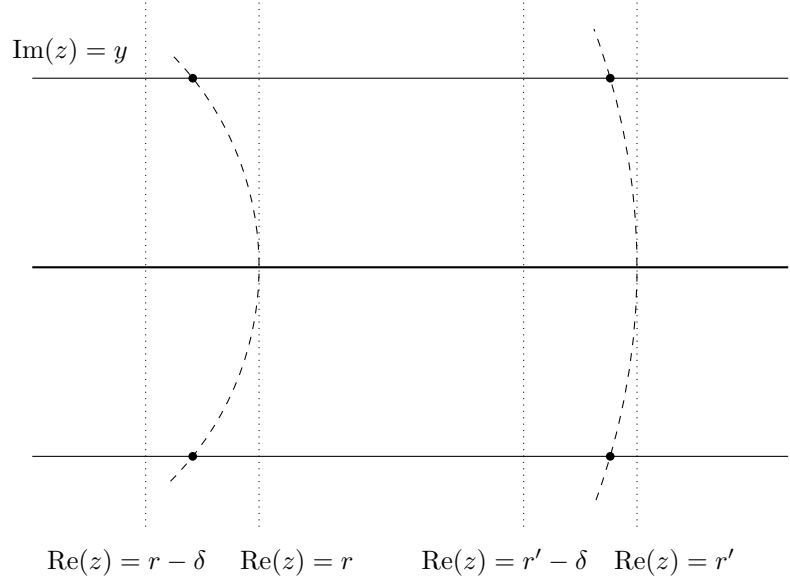
\begin{figure}
    \centering
    \begin{tikzpicture}[font=\small]
    
        \draw [thick] (-1,0) -- (9,0);        
        \draw (-1,2.5) -- (9,2.5);
        \draw node at (-0.5,2.85) {$\Imm(z) = y$};
        \draw (-1,-2.5) -- (9,-2.5);

        \centerarc[dashed](-2,0)(-45:45:4)

        \filldraw (1.1225, 2.5) circle (0.05);
        \filldraw (1.1225, -2.5) circle (0.05);

        \draw [dotted] (0.5,3.5) -- (0.5,-3.5);
        \draw node at (0.25,-3.9) {$\Ree(z)=r-\delta$};
        
        \draw [dotted] (2,3.5) -- (2,-3.5);
        \draw node at (2.5,-3.9) {$\Ree(z)=r$};

        \centerarc[dashed](-2,0)(-20:20.5:9)

        \filldraw (6.6458, 2.5) circle (0.05);
        \filldraw (6.6458, -2.5) circle (0.05);

        \draw [dotted] (7,3.5) -- (7,-3.5);
        \draw node at (7.5,-3.9) {$\Ree(z)=r'$};
        
        \draw [dotted] (5.5,3.5) -- (5.5,-3.5);
        \draw node at (5.25,-3.9) {$\Ree(z)=r'-\delta$};
    
    \end{tikzpicture}
    \caption{Two circles which are $\delta-$vertical at height $y$.}
    \label{fig:Vertical circle}
\end{figure}

\begin{rem}\label{rem:Vertical circle}
    Notice that the circle $\kappa(r)$ is $\delta$-vertical at height $y$ if and only if $$r \geqslant \frac{y^2+\delta^2}{2\delta}.$$ Thus, if the circle $\kappa(r)$ is $\delta$-vertical at height $y$ then for every $r'>r$, the circle $\kappa(r')$ is $\delta$-vertical at height $y$ too (see Figure \ref{fig:Vertical circle}).
\end{rem}

\begin{lem}\label{lem:Const for vertical}
    There exists $q>0$ such that for every $K\in\Z_+$ the following implication holds: if the circle $\kappa(r)$ is $1$-vertical at height $(2K+1)\pi$ for $r>1$, then the circle $\kappa(re^q)$ is $1$-vertical at height $(2(K+p) + 1)\pi$.
\end{lem}

\begin{proof}
According to \ref{rem:Vertical circle} we can write
\begin{equation}\label{eq:Const for vertical}
2r \geqslant (2K+1)^2\pi^2 + 1.
\end{equation}
To prove the thesis, we need to show that 
$$2re^q \geqslant (2(K+p)+1)^2\pi^2 + 1.$$
Using \eqref{eq:Const for vertical} it suffices to show that there exists $q>0$ such that for every $K\in\Z_+$
$$e^q \geqslant \frac{(2(K+p)+1)^2\pi^2 + 1}{(2K+1)^2\pi^2 + 1}.$$
Let us treat the expression on the right-hand side as elements of the sequence $(c_K)_{K\in\N}$. It is a decreasing sequence. Therefore, any $q$ satisfying $e^q \geqslant c_0$ is the desired number.
\end{proof}

\begin{lem}\label{lem:Vertical circles}
Provided $\zeta$ is large enough, the circle $\kappa(\E(a_n))$ is $1$-vertical at the height $(2M_{n+1} + 1)\pi$ for every $n\in\N$.
\end{lem}

\begin{proof}
We prove this by induction. For a fixed height $y > 0$ and any $\delta > 0$, there exists a sufficiently large $r$ such that the circle $\kappa(r)$ is $\delta$-vertical at height $y$. Therefore, by choosing $\zeta$ sufficiently large and noting that $a_0 = \zeta$, we can ensure that the circle $\kappa(\E(a_0))$ is $1$-vertical at any given height, especially at the height $(2M_1 + 1)\pi$.
Additionally, we can choose $\zeta$ such that for every $x\geqslant\zeta$, we have
$$\E(x) \geqslant x + q + 1,$$
where $q$ is the number obtained in Lemma \ref{lem:Const for vertical}.
Now, assume that the circle $\kappa(\E(a_n))$ is $1$-vertical at height $(2M_{n+1} + 1)\pi$. Then, based on how the numbers $a_n$ are defined, we have
$a_{n+1} \geqslant \E(a_n) - 1.$
Therefore
$a_{n+1} \geqslant a_n+q.$
So, we have
$\E(a_{n+1}) \geqslant \E(a_n)\cdot e^q.$
Thus, to show that the circle $\kappa(\E(a_{n+1}))$ is $1$-vertical at height $(2M_{n+2} + 1)\pi$, it suffices to show that the circle $\kappa(\E(a_n)\cdot e^q)$ is $1$-vertical at height $(2M_{n+2} + 1)\pi$. However, this follows from Lemma \ref{lem:Const for vertical} and the inductive assumption and the fact that $M_{n+1} + p = M_{n+2}$.
\end{proof}

\begin{figure}[h]
    \centering
    \begin{tikzpicture}[font=\small]
    
        \draw (-1,2) -- (9,2);
        \draw (-1,-2) -- (9,-2);
        
        \draw node at (-1.4,2) {\small $M_n$};

        \draw (-1,3.5) -- (9,3.5);
        \draw node at (-1.5,3.5) {\small $M_{n+1}$};
        \draw (-1,-3.5) -- (9,-3.5);

        \draw [dotted] (1,4) -- (1,-4);
        \draw node at (0.9,-4.3) {\small $a_n - 1$};
        
        \draw [dotted] (3,4) -- (3,-4);
        \draw node at (2.9,-4.3) {\small $b_{n+k}+1$};

        \draw [dotted] (6,4) -- (6,-4);
        \draw node at (5.9,-4.3) {\small $a_{n+1} - 1$};
        
        \draw [dotted] (8,4) -- (8,-4);
        \draw node at (7.9,-4.3) {\small $b_{n+k+1} + 1$};

        \fill [pattern={Lines[distance=1mm, angle=45, line width=0.1mm]}, pattern color=red, opacity = 0.5] (1,-2) rectangle (3,2);

        \draw node at (-0.5,-1) {\small $V(a_n,b_{n+k},M_n)$};
        \draw [-stealth](0.6,-1) -- (2,-1);

        \fill [pattern={Lines[distance=1mm, angle=45, line width=0.1mm]}, pattern color=blue, opacity = 0.5] (6,-3.5) rectangle (8,3.5);

        \draw node at (2.86,-3) {\small $V(a_{n+1},b_{n+k+1},M_{n+1})$};
        \draw [-stealth](4.5,-3) -- (7,-3);

        \centerarc[dashed](-5,0)(-21.5:21.5:10.5)
        \centerarc[dashed](-5,0)(-16:16:14)

        \draw node at (4.7,4.2) {\small $|z| = \E(a_n-1)$};
        \draw node at (9.2,4.2) {\small $|z| = \E(b_{n+k}+1)$};
        
    \end{tikzpicture}
    \caption{Rectangles $V(a_n,b_{n+k},M_n)$, $V(a_{n+1},b_{n+k+1},M_{n+1})$ and part of image of $V(a_n,b_{n+k},M_n)$.}
    \label{fig:E(V) covers V}
\end{figure}
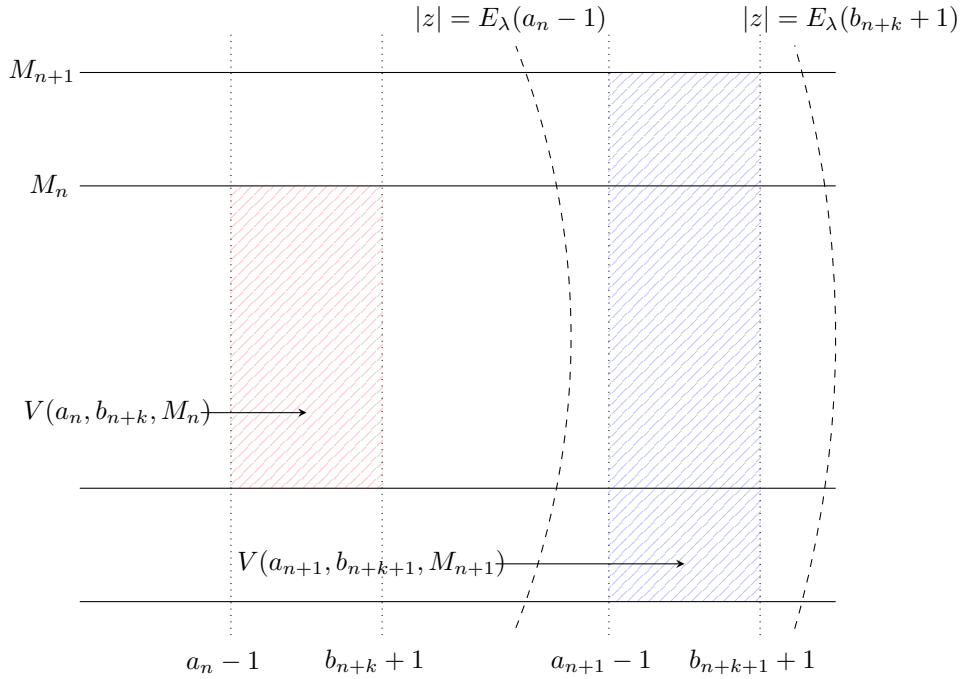

\begin{proof}[Proof of Proposition \ref{prop:E(V) covers V}]
Let us fix any $n\in\N$. The image of the set $V(a_n,b_{n+k},M_n)$ under the function $\E$ is an annulus centred at $0$, with an inner radius of $\E(a_n-1)$ and an outer radius of $\E(b_{n+k}+1)$. From Lemma \ref{lem:Vertical circles} we know that the circle $\kappa(\E(a_n))$ is $1$-vertical at height $(2M_{n+1} + 1)\pi$. Therefore, by invoking the construction of the numbers $a_n$, we deduce that 
$a_{n+1} \geqslant \E(a_n) - 1.$
Therefore,
$a_{n+1} - 1\geqslant \E(a_n) - 2.$
We can choose $\zeta$ large enough so that for every $x\geqslant\zeta$ we have 
$\E(x) \geqslant \E(x-1) + 2.$
Thus we obtain the inequality
$a_{n+1} - 1 \geqslant \E(a_n - 1).$
This means that the rectangle $V(a_{n+1},b_{n+k+1},M_{n+1})$ is contained within the set $\{z\in\C: |z| \geqslant \E(a_n - 1)\}$ (see Figure \ref{fig:E(V) covers V}).
To complete the proof, we need to show that the rectangle $V(a_{n+1},b_{n+k+1},M_{n+1})$ is contained within the closure of $B(0, \E(b_{n+k} + 1))$. Notice that 
$$\E(b_{n+k} + 1 ) = e\E(b_{n+k}) = e^2\E^{n+k+2}(\zeta) = e^2b_{n+k+1} - e^2.$$ For sufficiently large $\zeta$ we have 
$e^2x - e^2 \geqslant x + 2$
for every $x\geqslant\zeta.$
Therefore,
$e^2b_{n+k+1} - e^2 \geqslant b_{n+k+1} + 2.$
Thus, $\E(b_{n+k} + 1 ) -1 \geqslant b_{n+k+1} + 1.$ 
The circle $\kappa(\E(a_n))$ is $1$-vertical at height $(2M_{n+1} +1)\pi$. Since $a_n<b_{n+k}$, the circle $\kappa(\E(b_{n+k} + 1))$ is also $1$-vertical at height $(2M_{n+1} +1)\pi$. Therefore, the segment that forms the right side of the rectangle $V(a_{n+1}, b_{n+k+1}, M_{n+1})$, i.e. $\{z\in\C: \Ree(z) = b_{n+k+1} + 1,\textrm{ } |\Imm(z)| \leq (2M_{n+1} + 1)\pi\}$, is disjoint from the exterior of the circle $\kappa(\E(b_{n+k} + 1))$ (see Figure \ref{fig:E(V) covers V}),  which completes the proof.
\end{proof}

\begin{cor}\label{cor:Limit of a_n}
    We have $\lim_{n\to +\infty} a_n = +\infty$.
\end{cor}

\begin{definition}\label{def:Alpha^n}
    Let $\alpha_s^0 = \alpha_s$, which is the base of the tail $\omega_s$. For $n \geqslant 1$ let
    $$\alpha_s^n = \alpha_s^{n-1} \cup \E^n(\alpha_{\hat{s}^n}),$$
    where $\hat{s}^n$ is any element from the set $\sigma^{-n}(s)$ that belongs to $\Sigma_M^p$, e.g. $\hat{s}^n = 0_n s$, which is the itinerary of the form $(000\ldots 0 s_0 s_1 s_2 \ldots)$ -- $n$ zeros followed by the sequence $s$.
\end{definition}

\begin{lem}
    For every $s\in\Sigma_M^p$ and every $n\in\mathbb{N}$, we have $\alpha_s^{n} \subseteq \alpha_s^{n+1}$. Moreover, $\bigcup_{n\in\mathbb{N}} \alpha_s^n = \omega_s$.
\end{lem}

\begin{proof}
The property $\alpha_s^{n} \subseteq \alpha_s^{n+1}$ follows directly from the Definition \ref{def:Alpha^n}.
Using Remark \ref{rem:Alpha} we can write 
$$\alpha_s^0 = \gamma_s([\theta_s, F(\theta_{\hat{s}^1})))$$
and 
$$\E^n(\alpha_s) = \gamma_{\sigma^n(s)}([F^n(\theta_s), F^{n+1}(\theta_{\hat{s}^1})))$$
for every $s\in\Sigma_M^p$ and every $n\in\N$. Let us prove that we can represent $\alpha_s^n$ as 
$$\alpha_s^n = \gamma_s([\theta_s, F^{n+1}(\theta_{\hat{s}^{n+1}}))),$$
where $\hat{s}^n$ is as in Definition \ref{def:Alpha^n}. For $n=0$ we have $\alpha_s^0 = \gamma_s([\theta_s, F(\theta_{\hat{s}^1})))$. Let us assume that $\alpha_s^k = \gamma_s([\theta_s, F^{k+1}(\theta_{\hat{s}^{k+1}})))$ holds for some $k\in\N$. Then

$$\alpha_s^{k+1} = \gamma_s([\theta_s, F^{k+1}(\theta_{\hat{s}^{k+1}}))) \cup \E^{k+1}(\alpha_{\hat{s}^{k+1}}).$$
Since
\begin{align*}
    \E^{k+1}(\alpha_{\hat{s}^{k+1}}) & = \gamma_{\sigma^{k+1}(\hat{s}^{k+1})}([F^{k+1}(\theta_{\hat{s}^{k+1}}), F^{k+2}(\theta_{\hat{s}^{k+2}}))) \\
    & = \gamma_{s}([F^{k+1}(\theta_{\hat{s}^{k+1}}), F^{k+2}(\theta_{\hat{s}^{k+2}}))),
\end{align*}
we get $\alpha_s^{k+1} = \gamma_s([\theta_s, F^{k+2}(\theta_{\hat{s}^{k+2}})))$. 
Now it is enough to note that with $n\to +\infty$ we have $F^{n}(\theta_{\hat{s}^{n}}) \to +\infty$. This is because the point $F^{n}(\theta_{\hat{s}^{n}}) + 2\pi i s_0$ is close to  point $\gamma_s(F^{n}(\theta_{\hat{s}^{n}}))$ and $\gamma_s(F^{n}(\theta_{\hat{s}^{n}}))\in \E^n(\alpha_s) \subseteq \{z\in\C: \Ree(z)\geqslant a_n\}$ for every $s\in\Sigma_M^p$ and every $n\in\N$.    
\end{proof}

\begin{prop}\label{prop:Family P}
Let $k\in\N$ and let $s=(s_0s_1s_2 \ldots)\in\Sigma_M^p$. The family of sets
$$P_s^n = L_{\lambda, s_0} \circ L_{\lambda, s_1} \circ \ldots \circ L_{\lambda, s_{n-1}}\left(V(a_n, b_{n+k}, M_n)\right)$$
indexed by $n\in\Z_+$ is a decreasing family of subsets of $V(\zeta, b_k, M_0)$. Additionally,
$$\alpha_s^k\subseteq\displaystyle\bigcap_{n=1}^\infty P_s^n\subseteq \gamma_s.$$
\end{prop}

\begin{proof}
Directly from Proposition \ref{prop:E(V) covers V} we can see that the family $\{P_s^n\}_{n\in\Z_+}$ is a decreasing family of subsets of $V(\zeta, b_k, M_0)$. These sets are also closed as they are preimages under a continuous function of a closed set.
To show that 
$$\alpha_s^k\subseteq\displaystyle\bigcap_{n=1}^\infty P_s^n,$$
we need to prove that $\alpha_s^k\subseteq P_s^n$ for each $n\in\Z_+$. It is sufficient to show that $\E^n(\alpha_s^k) \subseteq V(a_n, b_{n+k}, M_n)$. We can decompose the curve $\alpha_s^k$ into pairwise disjoint sets $\beta_s^q = \alpha_s^q\setminus \alpha_s^{q-1}$, where $q\in\{1,2,\ldots, k\}$, and the set $\beta_s^0 = \alpha_s^0$. For each $q\in\{0,1,\ldots, k\}$ we have $\beta_s^q = \E^q(\alpha_{\hat{s}^q})$, where $\hat{s}^q$ is any element from the set $\sigma^{-q}(s)$ that belongs to $\Sigma_M^p$, e.g., $\hat{s}^q = 0_n s$. From Corollary \ref{cor:Alpha in V} we know that for every $n\in\Z_+$ the inclusion
$\E^{n+q}(\alpha_{\hat{s}^q}) \subseteq V(a_{n+q}, b_{n+q}, M_{n+q})$
holds.
Since the itinerary of elements in the set $\alpha_{\hat{s}^q}$ belongs to $\sigma^{-q}(s)$, the itinerary of elements in the set $\E^{n+q}(\alpha_{\hat{s}^q})$ is $\sigma^n(s)$. The itinerary $s$ is an element of $\Sigma_M^p$ so $\sigma^n(s)\in\Sigma_{M_n}^p$. Therefore, we can modify the above inclusion as
$\E^{n+q}(\alpha_{\hat{s}^q}) \subseteq V(a_{n+q}, b_{n+q}, M_n).$
Thus,
$$\E^n(\beta_s^q) =  \E^{n+q}(\alpha_{\hat{s}^q}) \subseteq V(a_{n+q}, b_{n+q}, M_n) \subseteq V(a_n, b_{n+k}, M_n).$$
Since the above inclusion holds for each $q\in\{0,1,\ldots, k\}$ we have the inclusion
$$\E^n(\alpha_s) \subseteq V(a_{n}, b_{n+k}, M_n)$$
for each $n\in\Z_+$. This completes the proof of the first inclusion.

Now, consider an arbitrary point $z$ belonging to the intersection of sets $P_s^n$. Such a point has the itinerary $s$ because the fact that $z\in P_s^n$ implies that the initial $n$ terms of the itinerary of $S(z)$ coincide with the initial $n$ terms of the itinerary $s$.
Such a point is an escaping point ($z\in\C$ such that $\E^n(z) \xrightarrow[]{n \to \infty} \infty$) because $\E^n(z)\in V(a_n,b_{n+k},M_n)$, and from Corollary \ref{cor:Limit of a_n} we have $a_n\to +\infty$ with $n\to + \infty$. Therefore, $z$ lies on $\gamma_s$.
\end{proof}

%-----------------------
\section{Accumulating on itself -- tangled hairs}\label{sec:Curly hairs}
%-----------------------

In this section we present a construction that allows us to create hairs that are, in some sense, tangled. For this reason, the accumulation set of these hairs turns out to be a topologically interesting set. The reasoning presented in this section follows that of \cite{DJ}. For the reader's convenience, we present it here. We present it with details.

\begin{definition}
Let $n\in\Z_+$ and $k,l\in\N$. We say that a curve passes twice through $V(a_n,b_{n+k}, M_l)$ if it connects the left and right sides of this rectangle at least twice. We say that a curve passes twice through the set
$$L_{\lambda, s_0} \circ L_{\lambda, s_1} \circ \ldots \circ L_{\lambda, s_{n-1}}\left(V(a_n, b_{n+k}, M_l)\right)$$
if its image under the function $\E^n$ passes twice through $V(a_n,b_{n+k}, M_l)$.
\end{definition}

Let $0_k s$ denote an itinerary that starts with a sequence of $k$ zeros, followed by an infinite sequence $s=(s_0s_1s_2\ldots)\in \Sigma_M^p$. Notice that the sequence $0_k s$ is also an element of the set $\Sigma_M^p$ thus for each $k\in\N$ there exists a tail $\omega_{0_k s,\zeta}$.

\begin{definition}
    Let $d > 0$. We say that a curve $\varphi\subseteq\mathbb{C}$ is $d$-distant from the real line if $\varphi\subseteq\{z\in\mathbb{C}: |\operatorname{Im}(z)| \leqslant d\}$.
\end{definition}

\begin{lem}\label{lem:Omega close to R}
    For each $\varepsilon>0$ there exists $K\in\Z_+$ such that for every $k\geqslant K$ and every itinerary $s=\in\Sigma_M^p$, the curve $\omega_{0_k s,\zeta}$ is $\varepsilon$-distant from the real line.
\end{lem}

\begin{proof}
    Let $z\in\C$ be a number such that $\Ree(z)\geqslant\zeta$, $z\in R_0$, $\E(z)\in R_0$, and $\Ree(\E(z)) > \Ree(z)$. Let us assume that $|\Imm(\E(z))| = r \leqslant \pi$. The point $\E(z)$ lies on a circle with radius $\lambda e^{\Ree(z)}$. For $\zeta$ large enough this circle intersects half-lines defined by $|\Imm(w)|= r$ and $\Ree(w)\geqslant 0$, at points with a real part greater than $\Ree(z)$, thus greater than $\zeta$. We obtain a constraint on the argument of $\E(z)$: $|\textrm{Arg}(\E(z))| \leqslant \textrm{arctan}\left(\frac{r}{\zeta}\right)$. In this way we obtain a constraint on the imaginary part of $z$: $|\Imm(z)| \leqslant \textrm{arctan}\left(\frac{r}{\zeta}\right)$. Therefore, if we assume that for $z$ we additionally have $\E^m(z)\in R_0$ for each $m\leqslant k-1$ we obtain a constraint
    \begin{equation}\label{eq:Omega close to R}
        |\Imm(z)| \leqslant f^{k-1}(\pi),
    \end{equation}
    where $f(r)=\textrm{arctan}\left(\frac{r}{\zeta}\right)$. As $k\to\infty$ the right-hand side of the inequality converges to $0$, because $\textrm{arctan}(x) < x$ for $x>0$, and $\textrm{arctan}(0) = 0$. Thus, for any $\varepsilon>0$ we can choose $K$ large enough such that for every $k\geqslant K$ we have $|\Imm(z)|\leqslant \varepsilon$. Now it suffices to notice that if $v\in\omega_{0_k s,\zeta}$ then $v$ satisfies the assumptions made in the analysis above. Therefore, for any $\varepsilon>0$ there exists $K$ large enough such that for every $k\geqslant K$ the entire curve $\omega_{0_k s,\zeta}$ is $\varepsilon$-distant from the real line. We complete the proof by observing that the \eqref{eq:Omega close to R} does not depend on $s\in\Sigma_M^p$.
\end{proof}

\begin{lem}\label{lem:L^n(Omega)}
   Let $z\in\C$ satisfy $\Ree (z) > 0$, $|\Imm(z)| < \varepsilon$, where $\varepsilon > 0$. If $\varepsilon$ is small enough, then there exists $n\in\N$ such that $\Ree(L_{\lambda,0}^n(z)) \leqslant 0$. There exists also the smallest $m\in\N$ such that $L_{\lambda,0}^m(z)\in\overline{B(0,1)}$. For every $k\leqslant m$ we have $|\Imm(L_{\lambda,0}^k(z))| \leqslant \varepsilon$. Moreover, there exists $m'\leqslant n$ such that $L_{\lambda,0}^{m'}(z)\in\overline{B(0,\frac{1}{e})}$.
\end{lem}

\begin{proof}
    Let $w$ be a point satisfying $\Ree(w) > 0$, $|w| > 1$. Since $|L_{\lambda, 0}'(w)| = \frac{1}{|w|}$ then $|L_{\lambda, 0}'(w)| \leqslant 1$. Thus, for such $w$ the distance from $L_{\lambda, 0}(w)$ to the projection of $L_{\lambda, 0}(w)$ onto the real line is no greater than the distance from $w$ to its projection onto the real line. Then if $|\Imm(w)| < \varepsilon$ for some $\varepsilon > 0 $ then $|\Imm(L_{\lambda,0}(w))| < \varepsilon$. Therefore, if there exist $m\in\N$ such that $|L_{\lambda, 0}^m(z)| \leqslant 1$ and for every $k\leqslant m$ we have $|L_{\lambda, 0}^k(z))| > 1$, then $|\Imm(L_{\lambda, 0}^k(z))| < \varepsilon$ for every $k\leqslant m$. Therefore, to find $m$ described in the thesis of this lemma, it is enough to prove that there exists $n\in\N$ such that $\Ree(L_{\lambda,0}^n(z)) \leqslant 0$. Now, let $v$ be a point satisfying $\Ree(v) > 1$ and $|\Imm(v)| < \varepsilon$. Let $v = x + iy$. Then
    $$ L_{\lambda, 0} (v) = \ln\left(\frac{1}{\lambda}\sqrt{x^2 + y^2}\right) + i \textrm{Arg}\left(\frac{v}{\lambda}\right).$$
    Since $|y| < \varepsilon$, we have
    $$\operatorname{Re}\left(v\right) < \ln\left(\frac{1}{\lambda}\sqrt{\zeta^2 + \varepsilon^2}\right) = \ln\left(\sqrt{\zeta^2 + \varepsilon^2}\right) - \ln(\lambda). $$
    Notice that for sufficiently small $\varepsilon$ the inequality
    $
    \ln\left(\sqrt{r^2 + \varepsilon^2}\right) - \ln(\lambda) < r
    $
    holds for every $r > 0$. This is because for every $r>0$ the inequality $\ln(r) + 1 \leqslant r$ holds and $\ln(\lambda) > -1$. Therefore, if $f(r) = \ln\left(\sqrt{r^2 + \varepsilon^2}\right) - \ln(\lambda)$ then for every $r > 0$ there exists $n\in\N$ such that $f^n(r) \leqslant 0$. Thus, we obtain $n\in\N$ such that $\Ree(L_{\lambda, 0}^n(z)) \leqslant 0$. We also obtain $m'\in\N$ such that $\Ree(L_{\lambda, 0}^{m'}(z)) < \frac{1}{e}$. Let $\eta = \varepsilon e^{m'-m}$. The point $L_{\lambda, 0}^{m'}(z)$ is $\eta$-distant from real line, because for every $z$ such that $|z|\geqslant \frac{1}{e}$ we have $|L_{\lambda, 0}(z)| \leqslant e$. Thus, for $\varepsilon$ small enough we have $|L_{\lambda, 0}^{m'}(z)| \leqslant \frac{1}{e}$.
\end{proof}

\begin{prop}\label{prop:Hair passes twice 1}
Let $n\in\Z_+$ and $j,l \in \mathbb{N}$. There exists an integer $K>0$ such that for every $k$ satisfying $k\geqslant K$ and for every $s\in\Sigma_M^p$, the hair corresponding to the itinerary $0_k s$ passes twice through the rectangle $V(a_n, b_{n+j}, M_l)$.
\end{prop}

\begin{proof}
   Let $n\in\Z_+$ and $j,l\in\N$. Let $\varepsilon > 0$. According to Lemma \ref{lem:Omega close to R} there exists $k\in \mathbb{Z}_+$ such that the curve $\mu = \omega_{0_k s,\zeta}$ is $\varepsilon$-distant from the real line. Let $w$ be the point on $\mu$ satisfying $\Ree(w)=\zeta$. Referring to Lemma \ref{lem:L^n(Omega)}, let $Q\in\N$ the smallest number such that $L_{\lambda,0}^{Q}(\mu) \cap \overline{B(0,1)} = \varnothing$. According to \ref{lem:L^n(Omega)} curve $L_{\lambda,0}^Q(\mu)$ is $\varepsilon$-distant from the real line. Again referring to Lemma \ref{lem:L^n(Omega)}, let $P\in\N$ be the smallest number such that $L_{\lambda,0}^{Q+P}(\mu) \cap \overline{B(0,\frac{1}{e})} = \varnothing$. The number $P$ is finite and depends only on $\lambda$ and $\varepsilon$. 
   Now, we create the curve $L_{\lambda,0}^{Q+P+1}(\mu)$. Since $\ln\left(\frac{1}{e}\right) = -1$ and $\ln(\lambda) > -1$ then, for sufficiently small $\varepsilon$, the curve $L_{\lambda,0}^{Q+P+1}(\mu)$ intersects the line $\{z\in\C: \Ree(z) = 0\}$. Let $\nu_0$ be a continuous subcurve of the curve $L_{\lambda,0}^{Q+P+1}(\mu)$ whose endpoint has a real part equal to $0$. For every $\delta >0$ we can choose $k>0$ such that the curve $\nu_0$ is $\delta$-distant from the real line. Moreover, the number $k$ does not depend on $s\in\Sigma_M^p$.
    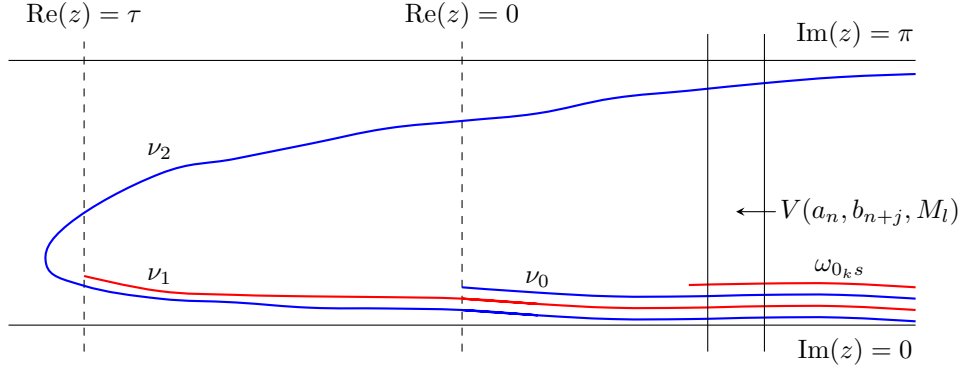
\begin{figure}[h]
        \centering
        \begin{tikzpicture}[font=\small]
    
        \draw (-2,2) -- (10,2);
        \draw (-2,-1.5) -- (10,-1.5);
        
        \draw node at (9.2,2.35) {$\Imm(z)=\pi$};    
        \draw node at (9.2,-1.85){$\Imm(z)=0$};

        \draw [dashed] (4,-1.85) -- (4,2.25);
        \draw [dashed] (-1,-1.85) -- (-1,2.25);

        \draw node at (4,2.6) {$\Ree(z)=0$};
        \draw node at (-1,2.6) {$\Ree(z)=\tau$};

        \draw [thick, blue] plot [smooth, tension=0.7] coordinates {(10,-1.45) (9,-1.4) (8,-1.4) (7,-1.42) (6,-1.42) (5,-1.37) (4,-1.30) (5,-1.37) (4,-1.30) (3,-1.28) (2,-1.27) (1,-1.20) (0,-1.15) (-1,-0.98) (-1.5,-0.7) (-1.2,-0.16) (0,0.5) (1,0.7) (2,0.9) (3, 1.1) (4,1.2) (5,1.3) (6,1.5) (7,1.6) (8,1.7) (9,1.78) (10,1.82)};

        \draw [thick, red] plot [smooth, tension=0.7] coordinates {(10,-1.30) (9,-1.25) (8,-1.25) (7,-1.27) (6,-1.27) (5,-1.22) (4,-1.15) (5,-1.22) (4,-1.15) (3,-1.13) (2,-1.12) (1,-1.10) (0,-1.05) (-1,-0.85)};

        \draw [thick, blue] plot [smooth, tension=0.7] coordinates {(10,-1.15) (9,-1.1) (8,-1.1) (7,-1.12) (6,-1.12) (5,-1.07) (4,-1)};
        
        \draw [thick, red] plot [smooth, tension=0.7] coordinates {(10,-1) (9,-0.95) (8,-0.95) (7,-0.97)};

        \draw node at (9,-0.75) {$\omega_{0_ks}$};
        \draw node at (5,-0.9) {$\nu_{0}$};
        \draw node at (0,-0.85) {$\nu_{1}$};
        \draw node at (0,0.8) {$\nu_{2}$};

        \draw (7.25,-1.85) -- (7.25,2.35);
        \draw (8,-1.85) -- (8,2.35);

        \draw node at (9.4,0) {\small $V(a_n,b_{n+j},M_l)$};
        \draw [-stealth](8.15,0) -- (7.625,0);
        
    \end{tikzpicture}
        \caption{The curve $\nu_2$ passes twice through $V(a_n,b_{n+j},M_l)$.}
        \label{fig:L^n(Omega)}
    \end{figure}
    Let us apply the transformation $L_{\lambda,0}$ again, this time to the curve $\nu_0$. Let $\tau < 0$. By choosing $\nu_0$ sufficiently close to the real line, we can obtain a continuous subcurve $\nu_1$ of the curve $L_{\lambda,0}(\nu_0)$ which endpoint has a real part equal to $\tau$, and it is arbitrarily close to the real line. At this point we apply the function $L_{\lambda,0}$ for the last time within the conducted construction. Notice that by choosing $\tau<0$ with a sufficiently large absolute value we obtain the curve $\nu_2=L_{\lambda,0}(\nu_1)$ which have a part with a sufficiently large real part and an imaginary part close to $\pi$ or $-\pi$. This is because points whose image under the function $\E$ lies far to the left and near the real axis must lie far enough to the right and close enough to the horizontal line determined by an odd multiplicity of $\pi$.    
    Therefore, for $\tau$ with sufficiently large absolute value the curve $\nu_2$ passes twice through the rectangle $V(a_n,b_{n+j}, M_l)$. In the end, it is enough to note that if, starting from the curve $\omega_{0_K s,\zeta}$, we are able to complete the above construction then, for each $k\geqslant K$, starting from the curve $\omega_{0_k s,\zeta}$ we are also able to complete the above construction.
\end{proof}

Based on the previous proposition, we can formulate another one as follows.

\begin{prop}\label{prop:Hair passes twice 2}
Let $n\in\Z_+$ and $j,l \in \mathbb{N}$ and let the numbers $s_0, s_1, \ldots, s_{n-1}$ be such that for every $q\in\{0,1,\ldots n-1\}$, $|s_q|\leqslant M_q$. There exists an integer $K>0$ such that for every $k$ satisfying $k\geqslant K$ and for every $r\in\Sigma_M^p$ the hair $\gamma_{s0_kr}$, where $s0_kr$ denotes the itinerary $s_0s_1\ldots s_{n-1}0_k r$, passes twice through the set
$$L_{\lambda, s_0} \circ L_{\lambda, s_1} \circ \ldots \circ L_{\lambda, s_{n-1}}\left(V(a_n, b_{n+j}, M_l)\right).$$
\end{prop}

\begin{proof}
Let $K$ be such a number that, according to Proposition \ref{prop:Hair passes twice 1}, for every $k\geqslant K$, the hair $\gamma_{0_k r}$ passes twice through the set $V(a_n, b_{n+j}, M_l)$. We can obtain the set $\gamma_{0_k r}$ as the image of $\gamma_{s 0_k r}$ under the map $\E^n$. Hence, $\gamma_{s0_k r}$ passes twice through the set
$$L_{\lambda, s_0} \circ L_{\lambda, s_1} \circ \ldots \circ L_{\lambda, s_{n-1}}\left(V(a_n, b_{n+j}, M_l)\right).$$
\end{proof}

%-----------------------
\section{Existence of indecomposable continua -- Proof of Theorem \ref{thm:A}}\label{sec:Continua}
%-----------------------

In this section we prove Theorem \ref{thm:A}, therefore we show that for some itineraries $s$, the set $I_s$ is an indecomposable continuum in the Riemann sphere.

Let
$$\beta_s = \gamma_s \cup \{\infty\}.$$ 
The curve $\beta_s$ can be parametrized as $\varphi: [0,+\infty) \longrightarrow \RS$ such that $\varphi(0) = \infty$. The set $I_s$ is contained in the strip $R_k$ for some $k\in\Z$. So there exists a point $z_0(s)\in\C$ and an open set $U\subseteq\C$ such that $z_0(s)\in U$ and $U \cap I_s = \varnothing$. Consider the set $\RS\setminus\{z_0(s)\}$ which is homeomorphic to the complex plane $\C$. Therefore, we can represent $\beta_s$ as a subset of $\RS\setminus\{z_0(s)\}$ and use Curry's Theorem (Theorem \ref{thm:Curry}).
Thus, the proof of Theorem \ref{thm:A} is preceded by proofs that for $s\in\Sigma_M^p$:
\begin{itemize}
    \item $\overline{\beta_s}$ is a $1$-dimensional continuum in $\RS\setminus\{z_0(s)\}$,
    \item $\overline{\beta_s}$ does not separate $\RS\setminus\{z_0(s)\}$.
\end{itemize}
Then we prove that for $s = 0_{n_0}t_{0}0_{n_1}t_{1}0_{n_2}t_{2}\ldots$, if lengths of blocks of zeros are large enough then:
\begin{itemize}
    \item $\beta_s$ accumulates on itself,
    \item $I_s\cup\{\infty\} = \overline{\beta_s}$.
\end{itemize}

 \begin{lem}\label{lem:1dim continuum}
    Let $s\in\Sigma_M^p$. The set $\overline{\beta_s}$ is a $1$-dimensional continuum in $\RS\setminus\{z_0(s)\}$.
\end{lem}

\begin{proof}
Firstly, we prove that for each $s\in\Sigma_M^p$, the set $I_s$ is closed. Let $s=(s_0s_1\ldots)$. Let $z$ be any point which is not an element of the set $I_s$. We want to find an open neighbourhood of the point $z$ that is disjoint from $I_s$. Let $S(z)=u=(u_0u_1\ldots)$. Since $z\notin I_s$, there exists $n\in\N$ such that $u_n\neq s_n$. The set $\E^n(I_s)$ is contained in $R_{s_n}$, and $\E^n(z)\in R_{u_n}$. If $\E^n(z)\in \intt R_{u_n}$ then there exists $\delta_1 > 0$ such that $\E^n(B(z,\delta_1))\subseteq \intt R_{u_n}$. Then $\E^n(B(z,\delta_1)) \cap \E^n(I_s) = \varnothing$, so $B(z,\delta_1) \cap I_s = \varnothing$. On the other hand, if $\E^n(z)\in\partial R_{u_n}$ then $z$ lies on the line determined by the equation $\Imm(w) = (2u_n+1)\pi$. Therefore, for every $m\in\N$ greater than $n$ we have $\Imm(\E^m(z)) = 0$. Since $s$ does not end in all zeros then there exists $q\in\N$ greater than $n$ such that $s_q\neq 0$. In that case, there exists $\delta_2>0$ such that $\E^q(B(z,\delta_2))\subseteq \intt R_{0}$. Since $s_q\neq 0$ we have $\E^q(I_s) \cap \E^q(B(z,\delta_2)) = \varnothing$. Therefore $B(z,\delta_2) \cap I_s = \varnothing$, which ends the proof that $I_s$ is closed. It follows that $\overline{\gamma_s}\subseteq I_s$.

The set $\overline{\beta_s}$ is non-empty. It is closed, so it is a compact subset of $\RS$ and $\RS\setminus\{z_0(s)\}$. It is also connected because it is the closure of a connected curve $\beta$. Therefore, $\overline{\beta_s} \subseteq \RS\setminus\{z_0(s)\}$ is a continuum.
We need to prove that the topological dimension of this set is $1$. To do this, it is enough to prove that $\overline{\gamma_s} \subseteq \C$ has a dimension equal to $1$. Referring to Theorem \ref{thm:Dimension 2}, since $\C$ is a metrizable and separable space it is enough to prove that the inductive dimension of $\overline{\gamma_s}$ is $1$. 
Therefore, according to Theorem \ref{thm:Dimension} it is sufficient to prove that the interior of $\overline{\gamma_s}\subseteq\C\simeq \R^2$ is empty. So assume that there exists a point $z$ and an open set $U$ such that $z\in U \subseteq \overline{\gamma_s}$. We have $\overline{\gamma_s}\subseteq I_s$, so $U\subseteq I_s$. Since the itinerary $s$ does not end in all zeros, the set $\bigcup_{n\in\N} \E^n(U)$ is disjoint from each line defined by $\Imm(w) = k\pi$, $k\in\Z$. However, the Julia set of the transformation $\E$ is the entire complex plane. Therefore, we obtain a contradiction with Montel's theorem (Corollary \ref{cor:Montel}). Thus, the set $\overline{\gamma_s}$ indeed has an empty interior, which completes the proof of this lemma.
\end{proof}

\begin{rem}
    In the proof of the lemma above, we only used the fact that $s$ does not end in all zeros and that $I_s$ is non-empty. Therefore, this lemma holds for itineraries $s$ from a larger set than $\Sigma_M^p$. 
\end{rem}

\begin{lem}\label{lem:Continuum 2}
    Let $s\in\Sigma_M^p$. The set $\overline{\beta_s}$ does not separate $\RS\setminus\{z_0(s)\}$.
\end{lem}
\begin{proof}
    It is enough to prove that $\overline{\gamma_s}$ does not separate the complex plane $\C$. Let us assume that the $\overline{\gamma_s}$ separates the plane $\C$. Then complement of the $\overline{\gamma_s}$ contains at least two connected components. One of these sets can be chosen in such a way that it is unbounded. Then another one must be contained in $R_{s_0}$ -- the strip containing $\gamma_s$. Let us denote this set by $U$.
    Now, we want to prove that $\E^n(U)\subseteq R_{s_n}$ for every $n\geqslant 1$. The boundary of $U$ is contained in $\overline{\gamma_s}$ so $\E^n(U)\cap R_{s_n}=\varnothing$. If $\E^n(U)\cap R_{k}=\varnothing$ for $k\neq s_n$ then $\E^n(U)$ intersects the line given by the equation $\Imm(w) = (2l+1)\pi$ for some $l\in\Z$. But then $E^n(\overline{\gamma_s})$ also intersects that line, which is a contradiction because every point on such a line has an itinerary which ends in all zeros. By the same reasoning we can conclude that $\E^n(U) \subseteq \intt R_{s_n}$ for all $n\geqslant 0$. But then the sum of all sets $\E^n(U)$ does not contain lines given by the equation $\Imm(w) = (2l+1)\pi$, $l\in\Z$. And this is a contradiction with Montel's theorem (Corollary \ref{cor:Montel}).
\end{proof}

\begin{lem}\label{lem:Continuum 3}
    Let $t_{0}, t_{1}, \ldots$ be a sequence of blocks of integers. Suppose that for every $m\in\N$ at least one term in the block $t_m$ is non-zero. There exists a sequence of natural numbers $(n_q)_{q\in\N}$ such that if
    $$s = 0_{n_0}t_{0}0_{n_1}t_{1}0_{n_2}t_{2}\ldots,$$
    then $\beta_s$ accumulates on itself.
\end{lem}

\begin{proof}
    Let $m_j\geqslant 1$ be the length of the block $t_{j}$. We choose any $M\in\N$ and $p\in\Z_+$.
    Then we modify the blocks $t_{j}$ for $j\geqslant 0$. Since each block $t_{j}$ has finite length, there exist natural numbers $c_j$ such that if we denote terms of the sequence $0_{c_j}t_{j}$ as $a_0, a_1, \ldots, a_{c_j+m_j-1}$, then for each $k\in\{0,1,\ldots, c_j + m_j - 1\}$ we have $|a_k|\leqslant M_k = M+kp$. Therefore, for every $N\geqslant 0$ the itinerary $s^N$ defined by
    $$s^N = 0_{c_N}t_{N}0_{c_{N+1}}t_{{N+1}}\ldots,$$
    is also an element of $\Sigma_M^p$. We have $s^N=\sigma^{\alpha(N)}(s^0)$ for some number $\alpha(N)$.
    By changing the notation by $0_{c_j}t_{j} \mapsto t_{j}$ we can consider, in the proof of this theorem, finite blocks $t_{j}$ that satisfy the following conditions:
    \begin{itemize}
        \item For every $j\geqslant 0$, the block $t_{j}$ contains at least one non-zero element,
        \item For every $N\geqslant 0$, the itinerary $s^N = t_{N}t_{{N+1}}t_{{N+2}}\ldots$ is an element of $\Sigma_M^p$.
    \end{itemize}
    Now we can proceed to add zeros to the itinerary $s^0$ so that the corresponding hair has the property of accumulating on itself. Let
    $$s=0_{n_0}t_{0}0_{n_1}t_{1}0_{n_2}t_{2}\ldots.$$ Let $j\in\N$. We want to determine the length of blocks of zeros. Let
    $$
    q_j = \sum_{p=0}^j m_p + \sum_{p=0}^{j} n_p,
    $$
    where we assume that the lengths of blocks $0_{n_0}$, $0_{n_1}$, $\ldots$, $0_{n_j}$ have already been determined. So $q_j$ is equal to the number of elements in the sequence $s$ that precede the block $0_{n_{j+1}}$.
    Now, let $\hat{s}_j(u)$ be an itinerary of the form
    $$
    \hat{s}_j(u) = 0_{n_0}t_{0}0_{n_1}t_{1}\ldots 0_{n_j}t_{j}0_{n_{j+1}}u,
    $$
    where $u\in\Sigma_M^p$. Let $n_{j+1}$ be a number such that the hair corresponding to the itinerary $0_{n_{j+1}}u$ passes twice through the rectangle $V(a_{q_j}, b_{2q_j}, M_{q_j})$. The existence of such a number is ensured by Proposition \ref{prop:Hair passes twice 1}. 
    According to Proposition \ref{prop:Hair passes twice 2}, the hair corresponding to the itinerary $\hat{s}_j(u)$ passes twice through the set
    $$
    L_{\lambda, s_0} \circ L_{\lambda, s_1} \circ \ldots \circ L_{\lambda, s_{q_j-1}}\left(V(a_{q_j}, b_{2q_j}, M_{q_j})\right),
    $$
    where $s_0$, $s_1$, $\ldots$, $s_{q_j-1}$ represent the first $q_j$ terms of the itinerary $\hat{s}_j(u)$. 
    This property holds for any $u\in\Sigma_M^p$. Thus, we can substitute $u$ with the itinerary $t_{{j+1}}0_{n_{j+2}}t_{{j+2}}\ldots$. This itinerary is from the set $\Sigma_M^p$ because it was obtained by adding zeros to the itinerary $s^N\in\Sigma_M^p$. 
    Thus, we obtain that the hair $\gamma_s$ passes twice through the set  
    $$
    L_{\lambda, s_0} \circ L_{\lambda, s_1} \circ \ldots \circ L_{\lambda, s_{q_j-1}}\left(V(a_{q_j}, b_{2q_j}, M_{q_j})\right)
    $$    
    for each $j\in\N$. Referring to Proposition \ref{prop:Family P} and the fact that $|\E'(z)|\geqslant \zeta > 1$ for every $z$ with $\Ree(z)\geqslant\zeta$, we get that $\gamma_s$ accumulates on any point of the set $\alpha_s^{q_j}$. Since this property holds for every $j\geqslant 0$ then we conclude that $\gamma_s$ accumulates on the entire tail $\omega_s$.

    To demonstrate that the set $\gamma_s$ accumulates on itself, not just on the tail $\omega_s$, let us consider the itineraries $\sigma^{p_k}(s)$, where $p_k = q_k + n_{k+1}$, $k\in \mathbb{N}$. This itinerary is of the form
    \begin{align*}
        \sigma^{p_k}(s) &= t_{m_{k+1}}0_{n_{k+2}}t_{m_{k+2}}0_{n_{k+3}}t_{m_{k+3}}\ldots.
    \end{align*}
    It also contains zero blocks. The block $0_{n_j}$ was constructed in such a way that the hair $\gamma_{0_{n_j} u}$ passes twice through the rectangle $V(a_{q_j}, b_{2q_j}, M_{q_j})$. Therefore, it passes twice through any rectangle $V(a_{l}, b_{2l}, M_{l})$ as long as $l\leqslant q_j$. Hence, the hair corresponding to the itinerary $\sigma^{p_k}(s)$ accumulates on its tail $\omega_{\sigma^{p_k}(s),\zeta_{p_k}}$. Thus, we obtain that the set $\gamma_s$ accumulates on the set
    \begin{align*}
        L_{\lambda, s_0} \circ L_{\lambda, s_1} \circ \ldots \circ L_{\lambda, s_{p_k-1}} \left( \omega_{\sigma^{p_k}(s),\zeta_{p_k}} \right),
    \end{align*}
    for every $k\geqslant 0$. Thus, the set $\gamma_s$ accumulates on itself.
    
    From the above construction, it follows that the set $\gamma_s$ also accumulates on the point $\infty$. Therefore, the curve $\beta_s$ accumulates on itself. 
\end{proof}

\begin{lem}\label{lem:Continuum 4}
    Let $t_{0}, t_{1}, \ldots$ be a sequence of blocks of integers. Suppose that for every $m\in\N$ at least one term in the block $t_m$ is non-zero. There exists a sequence of natural numbers $(n_q)_{q\in\N}$ such that if
    $$s = 0_{n_0}t_{0}0_{n_1}t_{1}0_{n_2}t_{2}\ldots,$$
    then $I_s\cup\{\infty\} = \overline{\beta_s}$.
\end{lem}
\begin{proof}
    According to Lemma \ref{lem:1dim continuum} it is enough to prove that every $z\in I_s$ is the element of the set $\overline{\gamma_s}$. Let $z\in I_s\setminus\gamma_s$.
    For $j\in\N$ let $d_j$ be such that $$\sigma^{d_j}(s) = e_j 0_{n_{j+1}} t_{{j+1}} 0_{n_{j+2}} t_{{j+2}}\ldots,$$
    where $e_j$ is the last non-zero element of the block $t_{j}$. Based on the construction presented in Section \ref{sec:Curly hairs} the curve $\gamma_{\sigma^{d_j + 1}(s)}$ intersects the disc $B(0,\varepsilon_j)$ and $\varepsilon_j$ is smaller the more zeros there are in the block $0_{n_{j+1}}$. Let us suppose that for every $j\in\N$ the radius $\varepsilon_j < \lambda$. Thus, every curve $\gamma_{\sigma^{d_j}(s)}$ intersects the half-plane $\{z\in\C: \Ree(z) < 0\}$. Let $\mu_j\in[0,+\infty)$ be the minimal real number such that the curve $\gamma_{\sigma^{d_j}(s)}$ is contained in the half-plane $\{z\in\C: \Ree(z)\geqslant -\mu_j\}$. The more zeros there are in the block $0_{n_{j+1}}$ the bigger is $\mu_j$. 
    For every $j\in\N$ one of two cases occurs: $\Ree(\E^{d_j}(z)) \geqslant -\mu_j - 2\pi$ or $\Ree(\E^{d_j}(z)) < -\mu_j - 2\pi$. Firstly, we want to show that for infinitely many $j\in\N$ the first condition holds. So let us assume that there exists $K$ such that for every $k\geqslant K$ we have $\Ree(\E^{d_k}(z)) < -\mu_j - 2\pi$.  Let us focus on points $\E^{d_k - 1}(z)$ and curves $\gamma_{\sigma^{d_k - 1}(s)}$. Let $v$ be the point on $\gamma_{\sigma^{d_k - 1}(s)}$ such that $\Ree(\E(v)) = -\mu_k$. Let us assume, for a moment, that $e_k>0$. We have $$\sqrt{\mu_k^2 + \pi^2(2e_k-1)^2} \leqslant |\E(v)| \leqslant \sqrt{\mu_k^2 + \pi^2(2e_k+1)^2}$$
    and 
    $$|\E^{d_k}(z)| \geqslant \sqrt{(\mu_k+2\pi)^2 + \pi^2(2e_k-1)^2}.$$
    Since $|\E(w)| = \lambda e^{\Ree(w)}$ for every $w\in\C$ then 
    \begin{equation}\label{eq:Closure}
        \Ree(\E^{d_k-1}(z)) > \Ree(v)
    \end{equation}
    if 
    $$(\mu_k+2\pi)^2 + \pi^2(2e_k-1)^2 > \mu_k^2 + \pi^2(2e_k+1)^2.$$ The above inequality is equivalent to the following one
    $$\mu_k > 2e_k\pi - \pi.$$ 
    If $e_k<0$ then, similarly, we get that $\mu_k > -2e_k\pi - \pi = 2|e_k|\pi - \pi$.
    For every $k\geqslant K$ we can set the length of the blocks $0_{n_{k+1}}$ large enough, so we can ensure that the above inequality, and therefore \eqref{eq:Closure}, holds for every $k\geqslant K$. 
    Since $v\in\gamma_{\sigma^{d_k - 1}(s)}$ and $\gamma_{\sigma^{d_k - 1}(s)}$ extends to the infinity to the right, then the distance between $\E^{d_k-1}(z)$ and the curve $\gamma_{\sigma^{d_k-1}(s)}$ is less than $2\pi$. 
    We have $|\Imm(\E^{d_k - 1}(z))| \geqslant \frac{\pi}{2}$, because even if $ \E^{d_k - 1}(z) \in R_0$ then $|\Imm(\E^{d_k - 1}(z))|\in \left(\frac{\pi}{2},\pi\right)$, because $\textrm{Arg}(\E^{d_k}(z))\in \left(\frac{\pi}{2},\frac{3\pi}{2}\right)\setminus\{ \pi \}$. Let $w_{k}\in\gamma_{\sigma^{d_{k}-1}(s)}$ be such that $|w_{k} - \E^{d_{k} - 1}(z)| < 2\pi$. 
    Let $\Phi_k$ be the inverse function of $E^{d_{k+1}-1 - d_k}$ which corresponds with the itinerary $s$ and which is defined on the set $B(\E^{d_{k+1} - 1}(z),r)$ for some $r \geqslant 2\pi$ ($\Phi_k$ is a composition of functions $L_{\lambda,q}$ for a proper choice of indexes $q$). Then, by Theorem \ref{thm:Koebe} and Lemma \ref{lem:Mis} we have
    $$|\Phi_k(w_{k+1}) - \E^{d_{k}}| = |\Phi_k(w_{k+1}) - \Phi_k(\E^{d_{k+1} - 1}(z))| \leqslant \frac{\frac{2}{\pi} \cdot 2\pi }{ (1-\frac{2\pi}{r})^2} = \frac{4}{(1-\frac{2\pi}{r})^2}.$$
    The greater is $k$ the greater are both $\mu_k$ and $|\E^{d_{k+1} - 1}(z)|$. Thus, by choosing $k\geqslant K$ large enough we can set $r$ in the above formula to be arbitrarily large. Therefore, there exists $k\geqslant K$ such that $|\Phi_k(w_{k+1}) - \E^{d_k}(z)| < 5$. But $\Phi_k(w_{k+1})$ is the element of the curve $\gamma_{\sigma^{d_k}(s)}$. Therefore, we have a contradiction with the assumption that $\Ree(\E^{d_k}(z)) < -\mu_k - 2\pi$. Thus, $\Ree(\E^{d_j}(z)) \geqslant -\mu_j - 2\pi$ for infinitely many $j\in\N$. 

    Let $U$ be an open neighbourhood of $z$. We want to show that $\gamma_s$ intersects $U$. Let $n_j$ be the increasing sequence of natural numbers such that $\Ree(\E^{d_{n_j}}(z))\geqslant -\mu_{n_j} - 2\pi$. Thus, the curve $\gamma_{\sigma^{d_{n_j}}(s)}$ intersects the disc $B(\E^{d_{n_j}}(z),4\pi)$. 
    For every $j\in\N$ let us represent the point $\E^{d_{n_j}}(z)$ as 
    $$\E^{d_{n_j}}(z) = \E^{d_{n_j} - d_{n_{j-1}}} \circ \E^{d_{n_{j-1}} - d_{n_{j-2}}} \circ \ldots \circ\E^{d_{n_1} - d_{n_0}}\circ\E^{d_{n_0}}(z).$$
    Using Lemma \ref{lem:Mis} there exists $V_j\subseteq B\left(z,\frac{4}{\pi^{j}}\right)$ such that $$\E^{d_{n_j}}(V_j) = B(\E^{d_{n_j}}(z),4\pi).$$ Therefore, there exists $j\in\N$ such that such $V_j$ is inside $U$, which means that $\gamma_s$ intersects $U$.
\end{proof}

\begin{proof}[Proof of Theorem \ref{thm:A}]
    Using Lemma \ref{lem:1dim continuum}, Lemma \ref{lem:Continuum 2} and Lemma \ref{lem:Continuum 3} we can apply Theorem \ref{thm:Curry}. Therefore, if lengths of blocks of zeros in $s$ are large enough, then the set $\overline{\beta_s}$ is an indecomposable continuum in $\RS\setminus\{z_0(s)\}$ and also in the Riemann sphere. According to Lemma \ref{lem:Continuum 4} we have $\overline{\beta_s} = I_s\cup\{\infty\}$. 
    To complete the proof, it is sufficient to show that $\overline{\beta_s}$ is equal to the accumulation set of the hair $\gamma_s$, which follows from Lemma \ref{lem:Continuum 3}.
\end{proof}

According to the above theorem, from any sequence of blocks $t_{0}, t_{1}, \ldots$, one can create an itinerary $s$ for which the closure of the set $\beta_s$ is an indecomposable continuum in the Riemann sphere. It is sufficient to insert blocks of zeros of appropriate length between consecutive elements of the sequence $t_{0}, t_{1}, \ldots$. The lengths of these blocks must form an unbounded sequence. It turns out that the lengths of these blocks cannot increase too slowly if we want to obtain a set $I_s$ which is an indecomposable continuum.

\begin{prop}\label{prop:Not continuum}
    For every $M\in\N$, every $p\in\Z_+$ and for every unbounded sequence $(n_k)_{k\in\N}$ of positive integers, there exists the itinerary $s\in\Sigma_M^p$ of the form
    $$s = 0_{n_0}t_{0}0_{n_1}t_{1}0_{n_2}t_{2}\ldots,$$
    such that $\overline{\beta_s} = \overline{\gamma_s} \cup \{\infty\}$ is not an indecomposable continuum in the Riemann sphere.
\end{prop}

The proof of the above theorem involves directly indicating such an itinerary. Firstly, we present one definition and one lemma. Let us recall that the function $F:\R\longrightarrow\R$ is defined by $F(t)=e^t-1$.

\begin{definition}
    The itinerary $s=(s_0s_1\ldots)$ is fast if 
   $$\forall_{x,A > 0} \hspace{1mm} \exists_{N\in\N} \hspace{1mm} \forall_{n\geqslant N} \hspace{1mm} \exists_{k\in\N} \hspace{1mm} |s_{n+k}| > A F^k(x).$$
\end{definition}

According to \cite[Corollary 6.9]{SZ}, if $s\in\Omega$ then $\gamma_s$ lands at the escaping point if and only if $s$ is a fast itinerary. On the other hand, if $\gamma_s$ lands at some point, then the closure of $\beta_s$ is not an indecomposable continuum.

\begin{lem}\label{lem:x and A > 2}
        If $x,A > 2$ and $y>Ax$ then $F^n(y)>AF^n(x)$ for all $n\geqslant 1$.
    \end{lem}
    \begin{proof}
        For $n=0$ the inequality $F^n(y)>AF^n(x)$ is of the form $y>Ax$, therefore it holds. Let us suppose that the inequality $F^n(y) > AF^n(x)$ holds for $y>Ax$ with $x,A>2$. Since $F(t)$ is increasing, then $F^{n+1}(y)>F(AF^n(x))$. For every $r>2$ we have $Ar > A + r$. Therefore
        $$F^{n+1}(y)>F(AF^n(x)) = e^{AF^n(x)}-1 > e^Ae^{F^n(x)}-1>Ae^{F^n(x)}-A = A(e^{F^n(x)}-1)=AF^{n+1}(x).$$
    \end{proof}

\begin{proof}[Proof of Proposition \ref{prop:Not continuum}]

Based on the above findings, it is enough to construct a fast itinerary $u\in\Sigma_M^p$ of the appropriate block form. Let $s=(s_0s_1\ldots)$ be such that $s_k = k$ for all $k\in\N$. Now we want to modify $s$ by inserting zero blocks $0_{n_0},0_{n_1},0_{n_2},\ldots$ in appropriate places. Let us recall that $0_k$ is a block of zeros of length $k$. We insert zero blocks into sequence $s$ as follows. Firstly, let $(s_{l_p})_{p\in\Z_+}$ be a subsequence of $s=(s_0s_1\ldots)$ such that $s_{l_p} \geqslant F^{n_p+p}(1)$ for every $p\in\Z_+$. Then, we insert the block $0_{n_0}$ at the beginning of $s$, and for every $p\in\Z_+$ we insert the block $0_{n_{p}}$ into $s$ between $s_{l_p}$ and $s_{l_p+1}$. In this way, we obtain an itinerary $u$ of the form 
$$u = 0_{n_0}t_{0}0_{n_1}t_{1}0_{n_2}t_{2}\ldots,$$
for which, for every $p\in\Z_+$, the last element of $t_{{p-1}}$ is $s_{l_p}$ which is not smaller than $F^{n_p+p}(1)$. Then, for every $p\in\Z_+$, the first element of $t_{{p}}$ is not smaller than $F^{n_p+p}(1)$ too.
The $p$-th block of zero has a length $n_p$, so the lengths of blocks of zeros are growing because the sequence $(n_p)_{p\in\N}$ is unbounded.

Note that if the condition from the definition of a fast itinerary is satisfied for some $x,A>0$, then it is satisfied for all $x',A'>0$ such that $x'<x$ and $A'<A$. Thus, let us choose $x,A>2$. We want to prove that there is $N\in\N$ such that for every $n\geqslant N$ there exists $k\in\N$ such that $|u_{n+k}| > AF^k(x)$.

    There exists $K\in\N$ such that $F^K(1) > Ax$. Let $N\in\N$ be such that $u_N$ is the last element of the block $t_{K}$. Therefore $u_N \geqslant F^{K+n_K}(1)$. Moreover, for every $n\geqslant N$ either $u_n = 0$ or $u_n \geqslant F^{K+n_K}(1)$. For $n \geqslant N$ such that $u_n \geqslant F^{K+n_K}(1)$ we have $u_{n+0} > F^K(1) > Ax = A F^0(x)$, so $k=0$ is the number we are looking for in this case. If $u_n =0$ then it is an element of the zero block $0_{n_p}$ for some $p\geqslant K+1$. Then there exists $k$ such that $u_{n+k}$ is the first element of block $t_{{p}}$. For such $k$ we have $u_{n+k} \geqslant F^{p+n_p}(1)$. We have $F^p(1)> F^K(1) > Ax$. According to Lemma \ref{lem:x and A > 2} we have $F^{p+k}(1) > A F^k(x)$. Since $k\leqslant n_p$ then $u_{n+k} > A F^k(x)$, so such $k$ is the number we are looking for in this case.

    Therefore, $u$ is a fast itinerary. Thus $\beta(u)$ does not land and the closure of $\beta(u)$ in the Riemann sphere is not an indecomposable continuum. We have $u\in\Sigma_0^1$ because $s\in\Sigma_0^1$, thus $u\in\Sigma_M^p$ for every $M\in\N$ and every $p\in\Z_+$.
\end{proof}

\begin{rem}
    We can ask a similar question for itineraries from the set $\Sigma_M^0$ i.e., itineraries bounded by constant $M$. However, all the itineraries from this set are not fast, so the above proof does not work in this case. In paper \cite{REM2}, L. Rempe shows a sketch of the proof of the theorem that there is a bounded itinerary of the form $0_{n_0}t_0 0_{n_1}t_1\ldots$ for which the set $I_s$ is not an indecomposable continuum.
\end{rem}

%-----------------------
\section{Dynamics -- proof of theorem B}\label{sec:Dynamics}
%-----------------------

In this section we focus on the dynamics in the set $I_s$. We prove Theorem \ref{thm:B}. Firstly, we show a simple observation.
\begin{lem}
    If $s$ is an unbounded itinerary then for every $z\in I_s$ the set $\{\Ree(w): w\in\Orb(z)\}\cap \R_+$ is unbounded and the orbit of $z$ is unbounded.
\end{lem}

\begin{proof}
    For every $z\in I_s$ the set $\Orb(z)$ is unbounded because $\E^n(z)\in R_{s_n}$ and the sequence $(s_n)_{n\in \N}$ is unbounded. Using the equality
    $$|\E^n(z)| = \lambda e^{\Ree(\E^{n-1}(z))}$$
    we get that the set $\{\Ree(w): w\in\Orb(z)\}\cap\R_+$ is unbounded.
\end{proof}
According to the above lemma, if $s$ is an unbounded itinerary, then for every $z\in I_s$ the $\omega$-limit set of $z$ contains $\infty$.

Now we proceed to the proof of Theorem \ref{thm:B}. First, we introduce a series of lemmas and propositions.

\begin{lem}\label{lem:Exp inequality}
    For $A,B > 0$ and $k\in\N$ there exists $N\in\N$ such that for every $n\geqslant N$ 
    $$\displaystyle{An^k + B\sum_{k=0}^{n}\E^k(0) <\E^{n+1}(0).}$$
\end{lem}
\begin{proof}
    First, let us prove that for any $C>0$ there exists $N$ such that for every $n\geqslant N$
    \begin{equation}\label{eq:Exp inequality 1}
        \displaystyle{C\sum_{k=0}^{n}\E^k(0) <\E^{n+1}(0).}
    \end{equation}
    Let the sequence $k_n$ be given by
    $$k_n = \displaystyle{\E^{n+1}(0) - C\sum_{k=0}^{n}\E^k(0).}$$
    We want to prove that there exists $N$ such that for every $n\geqslant N$ we have $k_n > 0$. To achieve this, we prove that for sufficiently large $n$ the inequality $k_{n+1} > k_n + 1$ holds.
    Using the definition of $k_n$, and reducing some terms, we transform this inequality into
    $$\E^{n+2}(0) > (1+C)\E^{n+1}(0) + 1,$$
    which is equivalent to the inequality
    $$\lambda e^{\E^{n+1}(0)} - (1+C)\E^{n+1}(0) - 1 > 0.$$
    Since there exists $a$ such that the real function $g(x) = \lambda e^x - (1+C)x -1$ is positive in the interval $(a,\infty)$, the above inequality is satisfied for sufficiently large $n$. Therefore, for $n$ large enough $k_{k+1} > k_n + 1$. Hence, the sequence $k_n$ is strictly increasing in the interval of the form $[N',\infty)$, and $\lim_{x\to +\infty}k_n = +\infty$. Thus, there exists the desired $N$ such that for every $n\geqslant N$ the inequality $k_n > 0$ holds.

    Now, let us prove that for any $D>0$ and $k\in\N$ there exists $N$ such that for every $n\geqslant N$
    \begin{equation}\label{eq:Exp inequality 2}
        Dn^k<\E^{n+1}(0).
    \end{equation}
    There exists $N'$ such that for every $n\geqslant N'$ we have $Dn^k < \lambda e^n$. Thus, it is enough to prove that there exists $N$ such that for every $n\geqslant N$ the inequality $\lambda e^n < \E^{n+1}(0)$ holds. This inequality is equivalent to $e^n < e^{\E^n(0)}$, which is equivalent to $n<  \E^n(0)$. For $n$ large enough $\E^n(0)-\E^{n-1}(0) > 2$, so there exists $N$ such that for every $n\geqslant N$ we have $n< \E^n(0)$.

    Now, it is enough to sum inequalities \eqref{eq:Exp inequality 1} and \eqref{eq:Exp inequality 2} for $C=2B$ and $D=2A$.   
\end{proof}

For $n\in\N$ let
$$r_n = \E^n(0) - 1\quad\textrm{ and }\quad T_n = [r_n, r_{n+1}) \subseteq \R.$$
For $n\in\N$ and $j\in\N$ let
$$\displaystyle{\rho_{j,n} = \E\left(-\frac{\E^{n+1}(0)}{e}\right) e^{j+1}\prod_{k=1}^{j}\E^k(0).}$$

\begin{lem}\label{lem:Limit set 1.5}
    Provided $n$ is large enough, if $z\in \C$ is such that $\Ree(z) < -\E(r_n) + 1$ then for every $j\in\{0, 1, \ldots, n+1\}$ we have $\E^{j+1}(z)\in B(\E^{j}(0),\rho_{j,n})$ and  $\rho_{j,n}<1$. In particular, $\Ree(\E^{n+2}(z))\in T_{n+1}$.
\end{lem}

\begin{proof}
    First, we prove that for sufficiently large $n$ we have $\rho_{j,n} < 1$ for each $j\in\{0,1,\ldots, n+1\}$.
    The sequence $(\E^j(0))_{j\in\mathbb{Z}_+}$ is a real increasing sequence, and $\E(0) = \lambda$. Thus, for every $j\geqslant 1$, we have $e\cdot \E^j(0) \geqslant e\cdot\lambda > 1$. Therefore, we only need to prove that $\rho_{n+1,n} < 1$ which is equivalent to $\ln(\rho_{n+1,n}) < 0$. We have
    
    \begin{align*}
        \ln(\rho_{n+1,n}) & = \ln \left(\lambda\right) -\frac{\E^{n+1}(0)}{e} + n+2 + (n+1)\ln(\lambda) + \sum_{k=0}^{n} \E^k(0)\\
        & \leqslant (n+2)(\ln(\lambda) + 1) -\frac{\E^{n+1}(0)}{e} + 2\E^n(0).
    \end{align*}
    In the above calculation we use the inequality $\sum_{k=0}^{n} \E^k(0) \leqslant 2\E^n(0)$ which follows from Lemma \ref{lem:Exp inequality}
    Thus, if 
    $$ (n+2)(1+\ln(\lambda))  + 2\E^n(0) < \frac{\E^{n+1}(0)}{e},$$
    then $\rho_{n+1,n} < 1$, but for $n$ large enough the above inequality holds, which again follows from Lemma \ref{lem:Exp inequality}.
    
    Now, let us move on to the proof of the second part.
    We have $\Ree(\E(z)) < -\E(r_n)+1 = -\E^{n+1}(0)e^{-1} +1$. Thus $\E(z)\in B(0, \rho_{0,n})$, which is our base of induction. Now, assume that for some $j\in \{0,1,\ldots,n\}$ we have $\E^{j+1}(z)\in  B(\E^j(0), \rho_{j,n})$. Then
    \begin{align*}
        |\E(\E^{j+1}(z)) - \E(\E^j(0))| & \leqslant  |\E^{j+1}(z) - \E^j(0)| \cdot \max_{v\in B_{j,n}} |\E(z)|\\
        &= \lambda|\E^{j+1}(z) - \E^j(0)| \cdot\max_{z\in B_{j,n}}e^{\operatorname{Re}(z)},
    \end{align*}
    where $B_{j,n} = B(\E^j(0), \rho_{j,n})$. By the inductive assumption, $|\E^{j+1}(z) - \E^j(0)| < \rho_{j,n}$. We know that $\rho_{j,n} < 1$ for sufficiently large $n$. Thus, for such $n$,
    $$\max_{z\in B_{j,n}}e^{\operatorname{Re}(z)} < e^{\E^j(0) + 1}.$$
    Therefore,
    \begin{align*}
        |\E^{j+2}(z) - \E^{j+1}(0)| & <  \lambda \rho_{j,n} \cdot e^{\E^j(0) + 1} = e\rho_{j,n}\cdot \E^{j+1}(0)\\
        &= \rho_{j+1,n}.
    \end{align*}
    Thus $\E^{j+2}(z)\in B_{j+1,n}$, which completes the proof. At the end, the corollary $\Ree(\E^{n+2}(z))\in T_{n+1}$ follows from the statement just proven applied to $j=n+1$: $\E^{n+2}(z)\in B_{n+1,n} = B(\E^{n+1}(0), \rho_{n+1,n})$.
\end{proof}

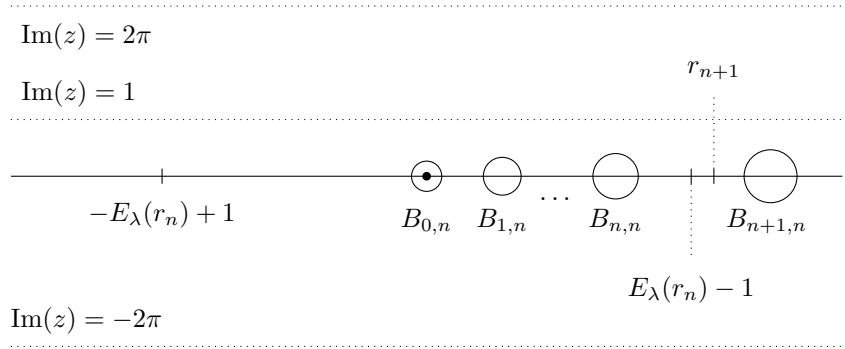
\begin{figure}[h]
        \centering
        \begin{tikzpicture}[font=\small]
    
            \draw [dotted] (-5.5,2.25) -- (5.5,2.25);
            \draw [solid] (-5.5,0) -- (5.5,0);
            \draw [dotted] (-5.5,-2.25) -- (5.5,-2.25);
            \draw [dotted] (-5.5,0.75) -- (5.5,0.75);

            \draw node at (-4.5,1.85) {$\Imm(z)=2\pi$};
            \draw node at (-4.5,-1.9) {$\Imm(z)=-2\pi$};
            \draw node at (-4.6,1.1) {$\Imm(z)=1$};

            \filldraw (0,0) circle (0.05);
            \draw node at (0,-0.6) {$B_{0,n}$};
            \draw node at (1,-0.6) {$B_{1,n}$};
            \draw node at (2.5,-0.6) {$B_{n,n}$};
            \draw node at (4.5,-0.6) {$B_{n+1,n}$};

            \centerarc[solid](0,0)(-180:180:0.2)
            \centerarc[solid](1,0)(-180:180:0.25)
            \centerarc[solid](2.5,0)(-180:180:0.3)
            \centerarc[solid](4.55,0)(-180:180:0.35)

            \draw [solid] (-3.5,0.1) -- (-3.5,-0.1);
            \draw node at (-3.5,-0.5) {$-\E(r_n)+1$};
            \draw [solid] (3.5,0.1) -- (3.5,-0.1);
            \draw [dotted] (3.5,-0.1) -- (3.5,-1.1);
            \draw node at (3.5,-1.5) {$\E(r_n)-1$};
            \draw [solid] (3.8,0.1) -- (3.8,-0.1);
            \draw [dotted] (3.8,0.1) -- (3.8,1.1);
            \draw node at (3.8,1.4) {$r_{n+1}$};
            \draw node at (1.7,-0.3) {$\ldots$};
            
        \end{tikzpicture}
        \caption{The disks $B_{j,n}$, $j\in\{0,1,\ldots,\}$. Their radii are smaller than $1$.} \label{fig:Images of left half-plane}
        \end{figure}
        
\begin{lem}\label{lem:Limit set 2}
    Provided $n$ is large enough, if $z\in \C$ is such that $\Ree(z)\in T_n$ for some $n\in \N$, $\E(z)\in R_k$ for some $k\in\Z$, $\kappa(\E(r_n))$ is $1-$vertical at height $(2|k|+1)\pi$ and $\Ree(\E(z)) < 0$, then for every $j\in\{0, 1, \ldots, n+1\}$ we have $\E^{j+2}(z)\in B(\E^{j}(0),\rho_{j,n})$ and $\rho_{j,n} < 1$. In particular, $\Ree(\E^{n+3}(z))\in T_{n+1}$.
\end{lem}

\begin{proof}       
    We have $\Ree(z) \geqslant r_n = \E^n(0) - 1$ and $\kappa(\E(r_n))$ is $1-$vertical at height $(2|k|+1)\pi$, where $k$ is such that $\E(z)\in R_k$. Thus $|\Ree(\E(z))| > \E(r_n)-1$. We have also $\Ree(\E(z)) < 0$, so $\Ree(\E(z)) < - \E(r_n)-1$. Using Lemma \ref{lem:Limit set 1.5} completes the proof.
\end{proof}

\begin{lem}\label{lem:Limit set 3}
    Provided $n$ is large enough, if $s\in\Sigma_M^p$, $z\in I_s$, $k\in\N$, $\Ree(\E^k(z))\in T_n$ and the circle $\kappa(\E(r_n))$ is $1-$vertical at height $(2M_{k+1} + 1)\pi$, then there exists $m \leqslant n+4$ such that $\Ree(\E^{k+m}(z))\in T_{n+1}$.
\end{lem}
\begin{proof}
    Let us consider the point $\E^{k+1}(z)$. There are two cases: $\Ree(\E^{k+1}(z)) < 0$ or $\Ree(\E^{k+1}(z)) > 0$. Suppose that $\Ree(\E^{k+1}(z)) < 0$. For $n$ large enough, according to Lemma \ref{lem:Limit set 2} we have \linebreak $\Ree(\E^{k+n+3}(z))\in T_{n+1}$. Now, let us consider the second case. The circle $\kappa(\E(r_n))$ is $1-$vertical at height $(2M_{k+1} + 1)\pi$. Thus, $\Ree(\E^{k+1}(z)) \geqslant \E(r_n)-1$. We have also $\Ree(\E^{k}(z)) < r_{n+1}$, so for $n$ large enough $\Ree(\E^{k+1}(z)) < r_{n+2}.$ Thus, $\Ree(\E^{k+1}(z))\in T_n$ or $\Ree(\E^{k+1}(z))\in T_{n+1}$. If $\Ree(\E^{k+1}(z))\in T_{n+1}$ then the number $m$ we are looking for is $m=1$. If $\Ree(\E^{k+1}(z))\in T_n$ then, based on Lemma \ref{lem:Const for vertical} and provided $n$ is large enough, the circle $\kappa(\E(\Ree(\E^{k+1}(z))))$ is $1$-vertical at height $(2M_{k+2} + 1)\pi$. Again we have two cases: $\Ree(\E^{k+2}(z)) <0$ or $\Ree(\E^{k+2}(z)) > 0$. Suppose that $\Ree(\E^{k+2}(z)) < 0$. According to Lemma \ref{lem:Limit set 2} we have $\Ree(\E^{k+n+4}(z))\in T_{n+1}$, similarly as in the case considered previously. Now, suppose that $\Ree(\E^{k+2}(z)) > 0$. For $n$ large enough $\E(r_n) > \E^n(0) + 1$ so $\Ree(\E^{k+1}(z)) > \E^n(0)$. Since $\kappa(\E(\Ree(\E^{k+1}(z))))$ is $1$-vertical at height $(2M_{k+2} + 1)\pi$, then $\Ree(\E^{k+2}(z)) > \E(\E^n(0)) - 1 = r_{n+1}$. Thus $\Ree(\E^{k+2}(z))\in T_{n+1}$.
\end{proof}

For $N,m\in\N$ let us define function $g(N,m)$. If $m\geqslant 0$ then 
$$g(N,m) = \sum_{j=0}^{m-1} (N+4+j) = m(N+4) + \frac{m(m-1)}{2}.$$ For $n\in\N$ let $y_n = (2M_n + 1)\pi = (2M + 2np + 1)\pi$.
\begin{lem}\label{lem:Function g}
    There exists $N\in\N$ such that for every $N'\geqslant N$ and $m\in\N$ the circle $\kappa(\E(r_{N'+m}))$ is $1-$vertical at height $y_{g(N',m)+1}$.
\end{lem}
\begin{proof}
    Firstly, let us note that if $m\geqslant 1$ and $\kappa(\E(r_{N+m}))$ is $1-$vertical at height $y_{g(N,m)+1}$, then $\kappa(\E(r_{(N+1)+(m-1)}))$ is $1-$vertical at height $y_{g(N+1,m-1)+1}$. This is because $g(N+1,m-1) < g(N,m)$. Therefore, it is enough to prove that there exists $N\in\N$ such that for every $m\in\N$ the circle $\kappa(\E(r_{N+m}))$ is $1-$vertical at height $y_{g(N,m)+1}$.

    Let us choose $N$ large enough such that for every $x\geqslant r_N$ we have $\E(x)\geqslant xe^q$, where $q$ is the constant with properties following from Lemma \ref{lem:Const for vertical}. Thus, according to Lemma \ref{lem:Const for vertical} it is enough to prove that there exists $N\in\N$ such that for every $m\in\N$ the circle $\kappa(r_{N+m})$ is $1-$vertical at height $y_{g(N,m)}$. 
    Let $N$ be large enough such that $r_N$ is $1-$vertical at height $y_{g(N,0)} = y_0$. Let us use induction. Assume that $r_{N+m}$ is $1$-vertical at height $y_{g_{(N,m)}}$. Thus, according to Remark \ref{rem:Vertical circle}, we have
    $$2r_{N+m} \geqslant y_{g_{(N,m)}}^2 + 1.$$
    We want to show that $r_{N+m+1}$ is $1$-vertical at height $y_{g_{(N,m+1)}}$, so it is enough to show that
    $$2r_{N+m+1} \geqslant y_{g_{(N,m+1)}}^2 + 1.$$
    We have
    $$g_{(N,m+1)} = g_{(N,m)} + (N+4+m).$$
    Thus
    $$y_{g_{(N,m+1)}} = y_{g_{(N,m)}} + 2p\pi(N+4+m).$$
    We apply this equality to the inequality we want to prove:
    \begin{align*}
        2r_{N+m+1} & \geqslant y_{g_{(N,m+1)}}^2 + 1 = \\
        & =  (y_{g_{(N,m)}}^2 + 1) + 4p^2\pi^2(N+4+m)^2 + 4p\pi(N+4+m)y_{g_{(N,m)}}.
    \end{align*}
    Using the inductive assumption, it is enough to prove that
    \begin{align*}
        2r_{N+m+1} - 2 r_{N+m}& \geqslant 4p^2\pi^2(N+4+m)^2 + 4p\pi(N+4+m)y_{g_{(N,m)}}.
    \end{align*}
    For $n$ large enough $r_{n+1}\geqslant 2r_n$ which follows from Lemma \ref{lem:Exp inequality}. Thus, we want to prove that 
    \begin{align*}
        r_{N+m} & \geqslant p^2\pi^2(N+4+m)^2 + \\
        & + 2p\pi^2(N+4+m)(2M+2pm(N+4) + pm(m-1) + 1),
    \end{align*}
    provided $N$ is large enough. The right side of the above inequality is the polynomial of the form
    $$C_0 + C_1 N + C_2 N^2 + C_3m + C_4 m^2 + C_5 m^3 + C_6 Nm + C_7 N^2m + C_8 Nm^2$$
    for some constants $C_j$, $j\in\{0,1,\ldots, 8\}$.
    Since $r_n = \E^n(0) - 1$, the above inequality is true for every $m\in\N$, provided $N$ is large enough.
    This completes the proof of the lemma.
\end{proof}

\begin{prop}\label{prop:Omega-limit set}
    Let $s\in\Sigma_M^p$. There exists $x>0$ such that for every $z\in I_s\setminus\gamma_s$ satisfying $\Ree(z) \geqslant x$ the $\omega$-limit set of $z$ is equal to $\Orb(0)\cup\{\infty\}$.
\end{prop}

\begin{proof}
Let $N$ be a number whose existence results from the above lemma. Let $x=r_{N+1}$, and let $z\in I_s\setminus\gamma_s$ be such that $\Ree(z)\geqslant x$. Let $N'\geqslant N + 1$ be such that $\Ree(z)\in T_{N'}$. Combining Lemma \ref{lem:Limit set 3} and Lemma \ref{lem:Function g} we obtain the increasing sequence of natural numbers $(n_j)_{j\in\N}$ such that $\Ree(\E^{n_j}(z))\in T_{N'+j}$. Since $z\notin\gamma_s$ then $\Ree(\E^n(z))\not\to \infty$, 
so we can choose the sequence of $n_j$ such that for infinitely many $n_j$ we have $\Ree(\E^{n_j + 1}(z)) <0$ or $\Ree(\E^{n_j + 2}(z)) <0$. Thus, according to Lemma \ref{lem:Limit set 2}, the set $\Orb(0)\cup \{\infty\}$ is a subset of the $\omega$-limit set of $z$. If $v\in \C\setminus\Orb(0)$, then let $d = \min\{|v - \E^n(0)|: n\in\N\}$. Then, there exists $K\in\N$ such that for every $k\geqslant K$, either $|\Ree(\E^k(z))| > |\Ree(v)| + 1$ or $\E^k(z)\in B(\E^m(0),\frac{d}{2})$ for some $m\in\N$. Thus, $v$ is not the element of the $\omega$-limit set of $z$.
\end{proof}

Let us introduce the following notation. This notation has already been partially used in the proof of Theorem \ref{thm:A}. For the itinerary $s$ of the form 
$$0_{n_0}t_{0}0_{n_1}t_{1}0_{n_2}t_{2}\ldots$$ let $b_j$ be the first non-zero element of the block $t_{j}$ and let $e_j$ be the last non-zero element of the block $t_{j}$. We can assume that $b_j$ is the first element of $t_{j}$ and that $e_j$ is the last element of $t_{j}$.
Let $a_j$ be such that
\begin{equation}\label{eq:ab}
    \sigma^{a_j}(s) = \underbrace{b_j\ldots e_j}_{t_{j}} 0_{n_{j+1}} t_{{j+1}} 0_{n_{j+2}} t_{{j+2}}\ldots
\end{equation}
and let $d_j$ be such that
\begin{equation}\label{eq:de}
    \sigma^{d_j}(s) = e_j 0_{n_{j+1}} t_{{j+1}} 0_{n_{j+2}} t_{{j+2}}\ldots.
\end{equation}
    
\begin{prop}\label{prop:At most one point}
    Let $t_{0}, t_{1}, \ldots$ be a sequence of finite blocks of integers, such that for every $m\in\N$ at least one term in block $t_m$ is non-zero.
    Then, there exists a sequence of natural numbers $(n_q)_{q\in\N}$ such that if
    $$s = 0_{n_0}t_{0}0_{n_1}t_{1}0_{n_2}t_{2}\ldots,$$
   then there exists at most one point in $I_s\setminus\gamma_s$ whose $\omega$-limit set is not a subset of $\Orb(0)\cup\{\infty\}$.
\end{prop}

\begin{proof}
    Let the blocks of zeros in the itinerary $s$ be large enough that this itinerary has the properties proven in the proof of Theorem \ref{thm:A}. Therefore, the itinerary $s$ is an element of the set $\Sigma_M^p$ for some $M$ and $p$. Let $x$ be a real number given by Proposition \ref{prop:Omega-limit set}. Therefore, if for $z\in I_s\setminus\gamma_s$ there exists $n\in\N$ such that $\sigma^n(s)\in\Sigma_M^p$ and $\Ree(\E^n(z))\geqslant x$, then $\omega$-limit set of $z$ is equal to $\Orb(0)\cup\{\infty\}$. Let $N$ be the smallest natural number such that $\E^N(0) - 1 > x $. The itinerary $s$ is of the form 
    $$s = 0_{n_0}t_{0}0_{n_1}t_{1}0_{n_2}t_{2}\ldots.$$
    The blocks of zeros in the itinerary $s$ are large enough that for every $j\in\N$ the itinerary 
    $$0_{n_j}t_{j}0_{n_{j+1}}t_{{j+1}}\ldots$$
    is an element of $\Sigma_M^p$. Let us modify every block of zeros by adding $N+2$ zeros to it. The itinerary $s$ is still of the form 
    $$s=0_{n_0}t_{0}0_{n_{1}}t_{{1}}\ldots.$$
    For every $j\in\N$ and every $l\in\{1,2,\ldots,N+2\}$ we have $\sigma^{d_j+1+l}(s)\in\Sigma_M^p$. Let $b>0$ be real number such that if $\Ree(w) < -b$ then $|\E^{N+1}(w) - \E^N(0)| < 1$. Thus, if  $\Ree(w) < -b$ then $\Ree(\E^{N+1}(w)) > x$. Let $a>0$ be a real number such that the circle $\kappa(\E(a))$ intersects the strip $R_0$ at points with a real value greater than $\max\{x,b\}$ or less than $-\max\{x,b\}$. Suppose that for some $z$ there exists $j\in\N$ such that $|\Ree(\E^{d_j}(z))|\geqslant\max\{a,b\}$. Then, there are two cases: $\Ree(\E^{d_j}(z))<-\max\{a,b\}$ or $\Ree(\E^{d_j}(z)) > \max\{a,b\}$. In the first case, we have $|\E^{d_j+N+1}(z) - \E^N(0)| < 1$, so $\Ree(\E^{d_j+N+1}(z)) > x $. But the itinerary $S(\E^{d_j+N+1}(z))\in\Sigma_M^p$. Thus, $\omega$-limit set of $z$ is equal to $\Orb(0)\cup\{\infty\}$. In the second case, we have $\Ree(\E^{d_j+1}(z)) > \max\{x,b\} \geqslant x$ or $\Ree(\E^{d_j+1}(z)) < -\max\{x,b\} \leqslant -b$. If $\Ree(\E^{d_j+1}(z)) > x$, then $\omega$-limit set of $z$ is equal to $\Orb(0)\cup\{\infty\}$, because $S(\E^{d_j+1}(z))\in\Sigma_M^p$. If $\Ree(\E^{d_j+1}(z)) < -b$ then $\Ree(\E^{d_j+N+2}(z))> x$, and again $\omega$-limit set of $z$ is equal to $\Orb(0)\cup\{\infty\}$, because $S(\E^{d_j+N+2}(z))\in\Sigma_M^p$. Therefore, if there exists $z$ whose $\omega$-limit set is not a subset of $\Orb(0)\cup\{\infty\}$ then for every $j\in\N$ we have $|\Ree(\E^{d_j}(z))| \leqslant c = \max\{a,b\}$. Let us denote by $G$ the set of every $z$ whose $\omega$-limit set is not a subset of $\Orb(0)\cup\{\infty\}$. Let us assume that $G$ is non-empty. Let $D_k = \{z\in \overline{R_k}: \Ree(z) \leqslant c\}$ for $k\in\Z$.
    For every $j\in\N$, and every $z\in G$ we have $\E^{d_j}(z)\in D_{e_j}$. Let $D$ be the diameter of every rectangle $D_k$. Let us represent the map $\E^{d_j}$ as 
    $$\E^{d_j} = \E^{d_{j} - d_{j-1}} \circ \E^{d_{j-1} - d_{j-2}} \circ \ldots \E^{d_{1} - d_{0}}\circ\E^{d_{0}}.$$
    For $j\geqslant 1$, let $\Phi_j$ be the inverse function of $\E^{d_j-d_{j-1}}$ which correspond with itinerary $s$ ($\Phi_j$ is a composition of functions $L_{\lambda,q}$ for a proper choice of indexes $q$), and let $\Phi_0$ be the inverse function of $\E^{d_0}$. Then 
    $$\displaystyle{G \subseteq \bigcap_{j\in\N} \Phi_0\circ \ldots \circ\Phi_j (D_{e_j})}.$$
    According to Lemma \ref{lem:Mis}, for every $j\in\N$ we have $|\Phi_j'| \leqslant \frac{1}{\pi}$ on $\overline{R_{e_j}}$. Thus, the diameter of $\Phi_0\circ \ldots \Phi_j (D_{e_j})$ is not greater than $\frac{D}{\pi^{j+1}}$. Therefore, the family of sets $\{\Phi_0\circ \ldots \circ\Phi_j (D_{e_j})\}_{j\in\N}$ is a decreasing sequence of non-empty, compact sets in $\C$ with diameters tending to $0$. Therefore, $G$ consists of only one element.
\end{proof}

\begin{definition}
        Let $n\in\N$. Let 
        $$\displaystyle{A^{+}(n) = \{z\in\C: |\Ree(z)| \leqslant\E(r_n)-1\textrm{, }\Imm(z)\in[0,\pi]\}\setminus \bigcup_{k=1}^{n+1}\E^k(H(n)),}$$
        where $H(n)=\{z\in\C:\Ree(z) < -\E(r_n)+1\}$. Similarly, let  
        $$\displaystyle{A^{-}(n) = \{z\in\C: |\Ree(z)| \leqslant\E(r_n)-1\textrm{, }\Imm(z)\in[-\pi,0]\}\setminus \bigcup_{k=1}^{n+1}\E^k(H(n)).}$$
    \end{definition}
    Both sets, $A^{+}$ and $A^{-}$, are closed. Now, we present a series of lemmas about $A^+(n)$. Analogous properties hold for $A^-(n)$.

    \begin{lem}\label{lem: Bus contraction}
       For every $n\in\N$ large enough there exists $m$ such that $L_{\lambda,0}^m(A^+(n))\subseteq \intt A^+(n)$.
    \end{lem}
    \begin{proof}
        Firstly, let us prove that, provided $n$ is large enough, we have $A^+(n) \subseteq \E(A^+(n)).$ The image of $A^+(n)$ under $\E$ is given by
        $$\E(A^+(n)) = \overline{B(0,\E(\E(r_n)-1))} \setminus \bigcup_{k=1}^{n+2}\E^k(H(n)).$$

        \begin{figure}[h]
        \centering
        \begin{tikzpicture}[font=\small]
    
            \draw [dotted] (-5.5,1) -- (-3.5,1);
            \draw [dotted] (3.5,1) -- (5.5,1);
            \draw [dotted] (-5.5,0) -- (-5,0);
            \draw [dotted] (5,0) -- (5.5,0);

            \filldraw (0,0) circle (0.05);
            \draw node at (0,-0.25) {$0$};
            \draw node at (-3.5,-0.75){$-\E(r_n)+1$};
            \draw node at (3.5,-0.75){$\E(r_n)-1$};

            \draw node at (-1,0.5){$A^+(n)$};
            \draw node at (-1,2.5){$\E(A^+(n))$};
    
            \centerarc[dashed](0,0)(0:180:5)
            \draw[dashed] (-5,0) -- (-3.5,0);
            \draw[dashed] (3.5,0) -- (3.95,0);
            \draw[dashed] (4.65,0) -- (5,0);

            \draw (-3.5,0) -- (-3.5,1); \draw (-3.5,0) -- (-3.5,1);
            \draw (-3.5,1) -- (3.5,1);
            \draw (3.5,0) -- (3.5,1);

            \draw[dotted] (-3.5,0) -- (-3.5,-0.5);
            \draw[dotted] (3.5,0) -- (3.5,-0.5);

            \draw (-3.5,0) -- (-0.2,0);
            \centerarc[solid](0,0)(0:180:0.2)
            \draw (0.2,0) -- (0.75,0);
            \centerarc[solid](1,0)(0:180:0.25)
            \draw (1.25,0) -- (2.2,0);
            \centerarc[solid](2.5,0)(0:180:0.3)
            \draw (2.8,0) -- (3.5,0);
            \centerarc[dashed](4.3,0)(0:180:0.35)
            
        \end{tikzpicture}
        \caption{The set of $A^+(n)$ and $\E(A^+(n))$.} \label{fig:Bus}
        \end{figure}
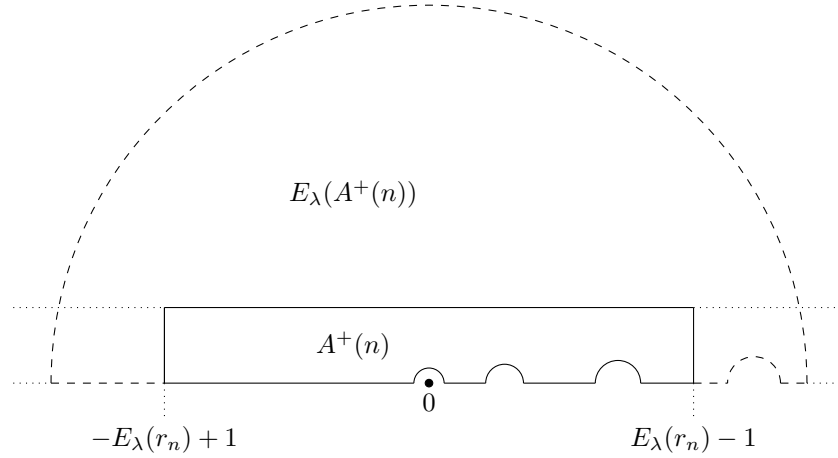
        Provided $n$ is large enough, $\E(\E(r_n)-1) > \E(r_n)-1 + \pi$ so 
        $$A^+(n)\subseteq B(0,\E(\E(r_n)-1)).$$
        Therefore, it is enough to show that $\E^{n+2}(H(n))\cap A^+(n) =\varnothing$.
        According to Lemma \ref{lem:Limit set 1.5}, provided $n$ is large enough, the set $\E^{n+2}(H)$ lies in the half-plan $\{z\in\C:\Ree(z)\geqslant r_{n+1}\}$. But $\E(r_n) - 1 < r_{n+1}$. Thus $A^+(n)\subseteq\E(A^+(n))$. Therefore, for every $m$ we have $L_{\lambda,0}^m(A^+(n))\subseteq A^+(n)$.

        To prove, that there exists $m$ such that $L_{\lambda,0}^m(A^+(n))\subseteq \intt A^+(n)$ let us suppose that $z\in\partial A^+(n)$. If $\Imm(z)=\pi$ or $\Imm(z) = 0$, then there exists $l$ such that $\E^l(z)\in\R$ and $\E^l(z) > \E(r_n)+1$, so $\E^l(z)\notin A^+(n)$. Therefore, such $z$ can not be an element of $L_{\lambda,0}^m(A^+(n))$ for $m\geqslant l$. If $\Ree(z) = -\E(r_n)-1$ or $z\in\partial \E^k(H(n))$ for $k\in\{0,1,\ldots,n+1\}$, then there exists $l$ such that $\E^l(z)\in\partial\E^{n+2}(H(n))$. According to Lemma \ref{lem:Limit set 1.5} such $z$ is not an element of $L_{\lambda,0}^m(A^+(n))$ for $m\geqslant l$. If $\Ree(z)=\E(r_n)-1$, then $|z| = \E(\E(r_n)-1)$, so $z\notin L_{\lambda,0}^m(A^+(n))$ for $m\geqslant 1$. Therefore, there exists $m$ such that $L_{\lambda,0}^m(A^+(n))\subseteq \intt A^+(n)$.

    \end{proof}

Let $q_+$ and $q_-$ be the repelling fixed points of $\E$ in $R_0$ with positive and negative imaginary parts respectively (see \cite{D}). These points are mentioned in the Theorem \ref{thm:B}.

    \begin{lem}\label{lem:Bus contraction 2}
        Provided $n$ is large enough, for every $j\in\Z_+$ there exists $K$ such that for every $k\geqslant K$ we have $L_{\lambda,0}^{k}(A^+(n))\subseteq B(q_+,\frac{1}{j})$.
    \end{lem}
    \begin{proof}
    Let $\rho_U$ be the hyperbolic metric on $U$ (see \cite{CG}). According to Lemma \ref{lem: Bus contraction} there exists $m$ such that $L_{\lambda,0}^m(A^+(n))\subseteq \intt A^+(n)$. Let us temporarily use the notation $f = L_{\lambda,0}^m$, $U=\intt A^+(n)$, $V = \intt f(A^+(n))$. Using Schwarz-Pick lemma (see \cite{CG}), for every $z\in U$ we have 
        $$|f'(z)|\frac{\rho_V(f(z))}{\rho_U(z)}\leqslant 1.$$
    Using again Schwarz-Pick lemma but for inclusion map from $V$ into $U$ we have $\rho_U(f(z)) < \rho_V(f(z))$. The strict inequality results from the fact that $V\neq U$. Therefore, for every $z\in U$ we have 
        $$|f'(z)|_{\rho_U} :=  |f'(z)|\frac{\rho_U(f(z))}{\rho_U(z)} < 1.$$
    The set $L_{\lambda,0}^m(A^+(n))$ is compact, therefore there exists $\mu\in(0,1)$ such that $|f'(z)|_{\rho_U} < \mu$ on $L_{\lambda,0}^m(A^+(n))$. Thus, $\mathrm{diam}(L_{\lambda,0}^{lm}(A^+(n)))\to 0$ with $l\to +\infty$. Therefore, 
    for every $j\in\Z_+$ there exists $N\in\Z_+$ such that for every $k\geqslant Nm$ we have $\mathrm{diam}(L_{\lambda,0}^{k}(A^+(n))) < \frac{1}{j}$. Moreover, the fixed point $q_+$ of $L_{\lambda,0}$ is an element of $L_{\lambda,0}^{k}(A^+(n))$. Therefore, for such $k$ we have $L_{\lambda,0}^{k}(A^+(n))\subseteq B(q_+,\frac{1}{j})$.
    \end{proof}
    \begin{rem}
        The analogous proof shows that, provided $n$ is large enough, for every $j\in\Z_+$ there exists $K$ such that for every $k\geqslant K$ we have $L_{\lambda,0}^k(A^-(n))\subseteq B(q_-,\frac{1}{j})$.
    \end{rem}

\begin{proof}[Proof of Theorem \ref{thm:B}]
    Let us suppose that $c>0$ is large enough that sets $L_{\lambda, j}(B(q_+,1))$ and $L_{\lambda, j}(B(q_-,1))$ are subsets of $\{z\in\C: \Ree(z)\in[-c,c]\}$ for every $j\in\Z$. Such $c$ is independent of $j$ because for some $U\subseteq\C$ we can obtain $L_{\lambda, j_2}(U)$ as a translation of $L_{\lambda, j_1}(U)$ by $2\pi(j_2-j_1)i$.
    Let $D_k = \{z\in \overline{R_k}: \Ree(z) \leqslant c\}$.
    Let $G$ be the set of $z\in I_s$ such that $\E^{d_j}(z)\in D_{e_j}$ for every $j\in\N$, where $d_j$ and $e_j$ are defined by \eqref{eq:ab} and \eqref{eq:de}. 
    
    Firstly, we want to show that $G$ is non-empty. We have
    $$\displaystyle{G = \bigcap_{j\in\N} \Phi_0\circ \ldots \circ \Phi_j (D_{e_j})},$$
    where $\Phi_j$ is the inverse function of $\E^{d_j-d_{j-1}}$ which corresponds with the itinerary $s$, as in the proof of Proposition \ref{prop:At most one point}.
    Every $D_{e_j}$ is a compact set. Therefore, to prove that $G$ is non-empty it is enough to show that $\Phi_0\circ \ldots \circ\Phi_j (D_{e_j})$ is non-empty for every $j\in\N$.
    Let us show that $\Phi_0\circ \ldots \circ\Phi_j (U)$ is non-empty for every non-empty $U\subseteq D_{e_j}$ and for every $j\in\N$.
    For $j=0$ this statement is true because $U$ is non-empty and therefore $\Phi_0(U)$ is non-empty for every non-empty $U\subseteq D_{e_0}$. Let us suppose that for some $j\in\Z_+$ the set $\Phi_0\circ \ldots \circ\Phi_{j-1} (U)$ is non-empty for every non-empty $U\subseteq D_{e_{j-1}}$. Let the block $t_j$ contained in the itinerary $s$ be given by
    $$t_{j} = (b_j\nu_j^0\nu_j^1\ldots \nu_j^{k_j}e_j),$$
    where $b_j\neq 0$.
    The set
    $$C_j = L_{\lambda,b_j}\circ L_{\lambda,\nu_j^1}\circ\ldots\circ L_{\lambda,\nu_j^{k_j}}(D_{e_j})$$
    is a compact subset of $R_{b_j}$. Let us assume that $b_j>0$. Therefore, there exists $n$ large enough that Lemma \ref{lem:Bus contraction 2} holds and that $C_j\subseteq \E(A^+(n))$. Therefore, $L_{\lambda,0}(C_j)\subseteq A^+(n)$. 
    \begin{figure}[h]
        \centering
        \begin{tikzpicture}[font=\small]
    
            \draw [dotted] (-5.5,1) -- (-3.5,1);
            \draw [dotted] (3.5,1) -- (5.5,1);
            \draw [dotted] (-5.5,0) -- (-5,0);
            \draw [dotted] (5,0) -- (5.5,0);

            \draw[dotted] (-5.5,2) -- (5.5,2);
            \draw[dotted] (-5.5,3) -- (5.5,3);

            \filldraw (0,0) circle (0.05);
            \draw node at (0,-0.25) {$0$};

            \draw node at (-1,0.5){$A^+(n)$};
            \draw node at (-1,3.75){$\E(A^+(n))$};
    
            \centerarc[dashed](0,0)(0:180:5)
            \draw[dashed] (-5,0) -- (-3.5,0);
            \draw[dashed] (3.5,0) -- (3.95,0);
            \draw[dashed] (4.65,0) -- (5,0);

            \draw (-3.5,0) -- (-3.5,1); \draw (-3.5,0) -- (-3.5,1);
            \draw (-3.5,1) -- (3.5,1);
            \draw (3.5,0) -- (3.5,1);

            \draw (-3.5,0) -- (-0.2,0);
            \centerarc[solid](0,0)(0:180:0.2)
            \draw (0.2,0) -- (0.75,0);
            \centerarc[solid](1,0)(0:180:0.25)
            \draw (1.25,0) -- (2.2,0);
            \centerarc[solid](2.5,0)(0:180:0.3)
            \draw (2.8,0) -- (3.5,0);
            \centerarc[dashed](4.3,0)(0:180:0.35)

            \draw node at (1,2.5){$C_j$};

            \draw plot [smooth, tension=0.7] coordinates {(1.5,2.5) (1.7, 2.7) (3.5, 2.8) (4.25, 2.4) (4, 2.2) (3, 2.5) (2, 2.3) (1.5,2.5)};
            
        \end{tikzpicture}
        \caption{The sets $C_j$ and $\E(A^+(n))$ for $n$ large enough.} \label{fig:Bus 2}
        \end{figure}
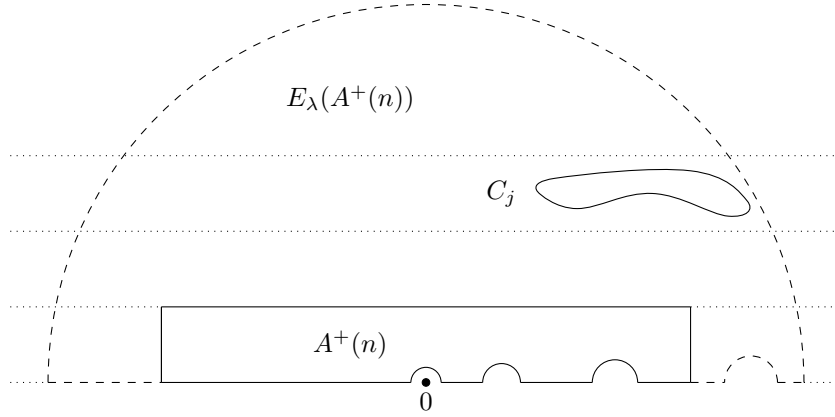
    According to Lemma~\ref{lem:Bus contraction 2} there exists $m_j$ such that $L_{\lambda,0}^{m_j}(C_j)\subseteq B(q_+,\frac{1}{j})$. If $b_j<0$ then, analogously, there exists $m_j$ such that $L_{\lambda,0}^{m_j}(C_j)\subseteq B(q_-,\frac{1}{j})$.
    Therefore, $$\Phi_j(D_{e_j})=L_{\lambda,e_{j-1}}\circ L_{\lambda,0}^{m_j}(C_j)\subseteq D_{e_{j-1}}.$$ 
    Let $V\subseteq D_{e_j}$ be non-empty. Therefore, $$\varnothing \neq \Phi_j(V)\subseteq\Phi(D_{e_j})\subseteq D_{e_{j-1}}.$$
    Therefore, the set $\Phi_0\circ \ldots \circ\Phi_{j-1} (\Phi_{j}(V))$ is non-empty for every non-empty $V\subseteq D_{e_j}$, which ends the inductive part of the proof.
    Therefore, if for every $j\in\N$ the $j$-th block of zeros in $s$ has length at least $m_j$ then
    $G$ is non-empty. 

    By our construction, for every $z\in G$ and every $j\in\Z_+$ the point $\E^{e_{j-1}+1}(z)$ is inside the disc $B(q_+,\frac{1}{j})$ or $B(q_-,\frac{1}{j})$. The disk where the point $\E^{e_{j-1}+1}(z)$ is located depends on the sign of $b_j$. Therefore, at least one of the points $q_+$, $q_-$ is the element of the $\omega$-limit set of $z$. 

    Now, let the blocks of zeros in the itinerary $s$ be large enough that we can use Proposition \ref{prop:At most one point}. Therefore, there exists exactly one point in $z_s\in I_s$ with $q_+$ or $q_-$ in $\omega$-limit set, and for every $z\in I_s\setminus(\gamma_s\cup\{z_s\})$ the $\omega$-limit set is $\Orb(0)\cup\{\infty\}$.
\end{proof}

%-----------------------
%% Bibliography
%-----------------------

\bibliographystyle{acm}
\bibliography{biblio.bib}

\end{document}